\documentclass[12pt,oneside]{article}

\usepackage[utf8]{inputenc}
\usepackage[english]{babel}

\usepackage[margin=1in]{geometry} 
\usepackage{amsmath,amssymb,amsthm}
\usepackage{bm,bbm,enumitem,hyperref,nccmath,ragged2e,titling}
\usepackage{graphicx,mathpazo}
\usepackage{mathtools}

\usepackage{xcolor}
\hypersetup{
    colorlinks,
    linkcolor={red!30!black},
    citecolor={red!30!black},
    urlcolor={red!30!black}
}

\newcommand{\R}{\mathbb{R}}
\renewcommand{\l}{\ell}

\newtheorem{theorem}{Theorem}
\newtheorem{lemma}[theorem]{Lemma}
\newtheorem{corollary}[theorem]{Corollary}
\newtheorem{definition}[theorem]{Definition}
\newtheorem{example}[theorem]{Example}

\newtheorem{proposition}[theorem]{Proposition}

\numberwithin{theorem}{section}
\numberwithin{equation}{section}

\usepackage{abstract}

\begin{document}
 
% --------------------------------------------------------------
%                         Start here
% --------------------------------------------------------------

%\begin{titlingpage}

\title{A Bochner Formula on Path Space \\
for the Ricci Flow}%replace X with the appropriate number
\author{Christopher Kennedy \\ %replace with your name
} %if necessary, replace with your course title
\date{}

\maketitle

%\end{titlingpage}

\begin{abstract}
\noindent
We generalize the classical Bochner formula for the heat flow on evolving manifolds $(M,g_{t})_{t \in [0,T]}$ to an infinite-dimensional Bochner formula for martingales on parabolic path space $P\mathcal{M}$ of space-time $\mathcal{M} = M \times [0,T]$. Our new Bochner formula and the inequalities that follow from it are strong enough to characterize solutions of the Ricci flow. Specifically, we obtain characterizations of the Ricci flow in terms of Bochner inequalities on parabolic path space. We also obtain gradient and Hessian estimates for martingales on parabolic path space, as well as condensed proofs of the prior characterizations of the Ricci flow from Haslhofer-Naber \cite{HN18a}. Our results are parabolic counterparts of the recent results in the elliptic setting from \cite{HN18b}.

%In the same vein that Haslhofer and Naber generalized infinite-dimensional gradient estimates to the time-dependent setting \cite{HN18a}, we obtain an analogous generalization of the infinite-dimensional Bochner inequality \cite{HN18b}. We will then see how the Bochner formula on $P\mathcal{M}$ for the Ricci flow in the time-dependent setting is related to the classical Bochner formula for the Ricci flow. Using this framework, we develop new characterizations of the Ricci flow, gradient and Hessian estimates for martingales on parabolic path space as well as condensed proofs of characterizations of solutions of the Ricci flow from \cite{Nab13}.
\end{abstract}


\tableofcontents

\newpage

\section{Introduction}

The goal of this paper is to prove a Bochner formula on path space for the Ricci flow, and to discuss some applications. This generalizes the Bochner formula on path space for Einstein metrics from Haslhofer and Naber \cite{HN18b}.
\newline \\
Throughout this paper, we shall use the convention that an evolving family of manifolds is a smooth and complete family of Riemannian manifolds $(M^{n},g_{t})_{t \in I}$ such that
\begin{align}
\sup_{M \times I} \left (|\mathrm{Rm}| + |\partial_{t}g|  + |\nabla \partial_{t}g| \right ) < \infty.    
\end{align}

\subsection{Background on Characterizations of Einstein Metrics}
To begin, let us recall some well-known characterizations of when a Riemannian manifold $(M,g)$ is a supersolution to the Einstein equations. Let $H_{t}f$ denote the heat flow of a function $f: M \rightarrow \R$. Then its gradient satisfies the Bochner formula
\begin{align}
(\partial_{t}-\Delta)|\nabla H_{t}f|^{2} = -2|\nabla^{2}H_{t}f|^{2} + 2\mathrm{Rc}(\nabla H_{t}f, \nabla H_{t}f).
\end{align}
Using this, an equivalence between supersolutions of the Einstein equations, the classical Bochner inequality and the gradient estimate readily follows, i.e.
\begin{align}
\mathrm{Rc} \geq 0 &\iff (\partial_{t}-\Delta)|\nabla H_{t}f|^{2} \leq -2|\nabla^{2}H_{t}f|^{2} \label{Bochner} \\
&\iff |\nabla H_{t}f| \leq H_{t} |\nabla f| \label{gradient},
\end{align}
for all test functions $f: M \rightarrow \R$.
\newline \\
Until recently, however, there was no analogous characterization of solutions to the Einstein equations. Such a characterization was discovered by Naber \cite{Nab13} by employing the analytic properties of path space $PM =C([0,\infty),M)$. This path space is naturally endowed with a family of Wiener measures $\{\mathbb{P}_{x}\}$ of Brownian motion starting at $x \in M$. One then introduces a notion of stochastic parallel transport and the corresponding family of parallel gradients $\{\nabla_{s}^{||} \}$. Using this foundation, Naber \cite{Nab13} developed an infinite-dimensional generalization of the gradient estimate $\eqref{gradient}$ to characterize solutions of the Einstein equations. Namely, he proved that
\begin{align}
\mathrm{Rc} = 0 \iff \left |\nabla_{x} \int_{PM} F\, d\mathbb{P}_{x} \right | \leq \int_{PM} |\nabla_{0}^{||}F|\, d\mathbb{P}_{x} \label{gradient inf},
\end{align}
for all test functions $F: PM \rightarrow \R$.
\newline \\
Interesting variants of these characterizations and estimates have been obtained in \cite{CT18a}, \cite{CT18b}, \cite{Wu16}, \cite{FW17} and \cite{WW18}.
\newline \\
Later, Haslhofer and Naber \cite{HN18b} proved an infinite-dimensional generalization of $\eqref{Bochner}$. Namely, they showed
\begin{align}
\mathrm{Rc}=0 \iff d|\nabla_{s}F_{t}|^{2} \geq \langle \nabla_{t}|\nabla_{s}F_{t}|^{2}, dW_{t} \rangle \label{Bochner inf}
\end{align}
for all martingales $F_{t}: PM \rightarrow \R$.
\newline \\
Using the infinite-dimensional Bochner formula $\eqref{Bochner inf}$, they gave a simpler proof of the infinite-dimensional gradient estimate $\eqref{gradient inf}$ in a similar vein to how the classical Bochner formula $\eqref{Bochner}$ readily implies the classical gradient estimate $\eqref{gradient}$.

\subsection{Background on Characterizations of Ricci Flow} \label{Background on Characterizations of Ricci Flow}

To motivate the characterization of solutions of the Ricci flow, let us first recall characterizations of supersolutions, namely evolving Riemannian manifolds $(M,g_{t})_{t \in I}$ such that
\begin{align}
\partial_{t}g_{t} \geq -2\mathrm{Rc}_{g_{t}}.
\end{align}
To begin, consider the heat flow $H_{st}f$ on this evolving background, namely the solution of the heat equation $\partial_{t}u = \Delta_{g_{t}}u$ with initial condition $f$ at time $t=s$. Then its gradient satisfies the Bochner formula
\begin{align}
(\partial_{t}-\Delta_{g_{t}})|\nabla H_{st}f|^{2} = -2|\nabla^{2}H_{st}f|^{2} + (\partial_{t}g_{t} + 2\mathrm{Rc}_{g_{t}})(\nabla H_{st}f, \nabla H_{st}f).
\end{align}
Using this, an equivalence between supersolutions of the Ricci flow, the Bochner inequality and the gradient estimate readily follows, i.e.
\begin{align}
\partial_{t}g_{t} \geq -2\mathrm{Rc}_{g_{t}} &\iff (\partial_{t}-\Delta_{g_{t}})|\nabla H_{st}f|^{2} \leq -2|\nabla^{2}H_{st}f|^{2} \label{Bochner 2} \\
&\iff |\nabla H_{st}f| \leq H_{st}|\nabla f|, \label{gradient 2}
\end{align}
for all test functions $f: M \rightarrow \R$.
\newline \\
To generalize the inequality $\eqref{gradient 2}$ to an infinite dimensional estimate, Haslhofer and Naber \cite{HN18a} considered space-time $\mathcal{M} = M \times I$ equipped with the space-time connection defined on vector fields by
\begin{align}
\nabla_{X}Y = \nabla_{X}^{g_{t}}Y, \hspace{5em} \nabla_{t}Y = \partial_{t}Y + \frac{1}{2} \partial_{t}g_{t}(Y, \cdot)^{\#g_{t}}  \label{10}
\end{align}
The main difference, compared to the infinite dimensional estimate that characterizes Einstein metrics, is that the parabolic path space $P_{T}\mathcal{M}$ only consists of continuous space-time curves $\{\gamma_{\tau}=(T-\tau,x_{\tau})\}$ that move backwards along the time-axis with unit speed and start at fixed time $T \in I$. This path space is naturally endowed with a family of parabolic Wiener measures $\{\mathbb{P}_{(x,T)} \}$ of Brownian motion starting at $(x,T) \in \mathcal{M}$ and parabolic stochastic parallel gradients $\{\nabla_{\sigma}^{||}\}_{\sigma \geq 0}$ defined via $\eqref{10}$. Using this framework, Haslhofer and Naber proved an infinite-dimensional generalization of the gradient estimate $\eqref{gradient 2}$ that characterizes solutions of the Ricci flow. Namely, they proved that
\begin{align} \label{gradient 2 inf}
\partial_{t}g_{t} = -2\mathrm{Rc}_{g_{t}} \iff \left |\nabla_{x} \int_{P_{T}\mathcal{M}} F\, d\mathbb{P}_{(x,T)} \right | \leq \int_{P_{T}\mathcal{M}} |\nabla_{0}^{||}F| \, d\mathbb{P}_{(x,T)}
\end{align}
for all test functions $F: P_{T}\mathcal{M} \rightarrow \R$.
\newline \\
Some nice variants of these characterizations have been obtained by Cheng and Thalmaier \cite{CT18b}. Moreover, Cabezas-Rivas and Haslhofer \cite{CH19} found an interesting link between estimates in the elliptic and parabolic settings.
\newline \\
However, there is no analogous treatment of the Bochner inequality $\eqref{Bochner inf}$  in the time-dependent setting. The primary goal of this paper shall be to prove such an equivalent notion.

\subsection{Bochner Formula on Parabolic Path Space}

Let $(M,g_{t})_{t \in I}$ be a family of evolving manifolds and let $\mathcal{M}= M \times I$ be its space-time equipped with the space-time connection defined on vector fields via $\eqref{10}$. Next, as in Section 1.2, we consider the parabolic path space $P_{T}\mathcal{M}$, given by
\begin{align}
P_{T}\mathcal{M} := \left \{(x_{\tau},T-\tau)_{\tau \in [0,T]} | x \in C([0,T], M) \right \}, \label{11}
\end{align}
and endow this space with the parabolic Wiener measure of Brownian motion on space-time, $\mathbb{P}_{(x,T)}$, based at $(x,T) \in \mathcal{M}$ as well as the associated parabolic parallel gradients $\nabla_{\sigma}^{||}$ defined via stochastic parallel transport on space-time $\mathcal{M}$. To explain these notions in more detail, first recall that the solution to the heat equation $\partial_{t}u = \Delta_{g_{t}}u$ with initial condition $f$ at time $t=s$ is given by convolving with the heat kernel i.e.
\begin{align}
H_{st}f(x) = \int_{M} H(x,t|y,s) f(y)\, dV_{g_{s}}(y).
\end{align}
The Wiener measure $\mathbb{P}_{(x,T)}$ is then uniquely characterized in terms of the heat kernel by
\begin{align}
&\mathbb{P}_{(x,T)} \left [X_{\tau_{1}} \in U_{1}, \dots , X_{\tau_{k}} \in U_{k} \right] \\
&\, = \int_{U_{1}} \! \! \scalebox{0.75}[1.0]{\(\cdots\)} \! \! \int_{U_{k}} H(x, T | x_{1}, T \scalebox{0.75}[1.0]{\(-\)} \tau_{1}) \scalebox{0.75}[1.0]{\(\cdots\)} H(x_{k-1}, T \scalebox{0.75}[1.0]{\(-\)} \tau_{k-1} | x_{k}, T \scalebox{0.75}[1.0]{\(-\)} \tau_{k}) d\mathrm{Vol}_{g_{T \scalebox{0.75}[1.0]{\(- \)} \tau_{1}}}(x_{1}) \scalebox{0.75}[1.0]{\(\cdots\)} d\mathrm{Vol}_{g_{T \scalebox{0.75}[1.0]{\(- \)} \tau_{k}}}(x_{k}) \nonumber
\end{align}
where $X_{\tau}$ is a Brownian motion on $M$ starting at $x$. Moreover, the stochastic parallel gradient $\nabla_{\sigma}^{||}F(\gamma) \in (T_{x}M,g_{T})$ of a function $F:P_{T}\mathcal{M} \rightarrow \R$, is expressed in terms of the Fr\'echet derivative by
\begin{align}
D_{V^{\sigma}}F(\gamma) = \langle \nabla_{\sigma}^{||}F(\gamma), v \rangle_{(T_{x}M,g_{T})},
\end{align}
where $V^{\sigma}$ is the vector field along $\gamma$ defined by $V_{\tau}^{\sigma} = P_{\tau}^{-1}v \mathbbm{1}_{[\sigma, T]}(\tau)$  and $\{P_{\tau}\}$, a family of isometries, referred to as stochastic parallel transport.
\newline \\
\noindent With the aim of generalizing $\eqref{Bochner 2}$ to an infinite-dimensional estimate, we consider martingales on parabolic path space, i.e. $\Sigma_{\tau}$-adapted integrable processes $F_{\tau}: P_{(x,T)}\mathcal{M} \rightarrow \R$ that satisfy
\begin{align}
F_{\tau_{1}} = \mathbb{E}[F_{\tau_{2}} | \Sigma_{\tau_{1}} ],
\end{align}
where $\mathbb{E}[ \cdot \, | \, \Sigma_{\tau}]$ denotes the conditional expectation with respect to the $\sigma$-algebra $\Sigma_{\tau}$ of events observable at time $\tau$.
%on the parabolic path space of space-time. We shall take a test function $F: P\mathcal{M} \rightarrow \R$ where $F \in L^{1}(P\mathcal{M},d\mathbb{P}_{(x,T)})$ and induce a martingale $F_{\tau}(\gamma)$
%\begin{align}
%F_{\tau}(\gamma) = \int_{P_{\gamma_{\tau}}\mathcal{M}} F(\gamma |_{[0, \tau]} * \gamma')\, d\mathbb{P}_{\gamma_{\tau}}(\gamma').
%\end{align}
\newline \\
For example, if $F(\gamma)=f(\pi_{1}\gamma_{\tau_{1}})$, where $f: M \rightarrow \R$ and $\pi_{1}: M\times I \rightarrow I$, then the induced martingale $F_{\tau} = \mathbb{E}[F\, |\, \Sigma_{\tau}]$ for $\tau < \tau_{1}$ is given by
\begin{align}
F_{\tau}(\gamma) = H_{T-\tau_{1},T-\tau}f(\pi_{1} \gamma_{\tau}) \qquad (\mathrm{see\, example\, \ref{example 2.18}}) \label{one-point heat flow}.
\end{align}
Specifically, martingales generalize heat flow. This analogue between martingales and heat kernels will motivate our development of the following generalized Bochner formula on $P\mathcal{M}$.
\begin{theorem} (Generalized Bochner Formula on $P\mathcal{M}$) \label{Bochner gen}
Let $F_{\tau} : P_{(x,T)}\mathcal{M} \rightarrow \R$ be a martingale on the parabolic path space of space-time. If $\sigma \geq 0$ is fixed, then 
\begin{align}
d(|\nabla_{\sigma}^{||}F_{\tau}|^2) &= \langle \nabla_{\tau}^{||}|\nabla_{\sigma}^{||}F_{\tau}|^2,dW_{\tau} \rangle + (\dot{g}+2\mathrm{Rc})_{\tau}(\nabla_{\tau}^{||}F_{\tau},\nabla_{\sigma}^{||}F_{\tau})\, d\tau \nonumber \\
&\quad + 2|\nabla_{\tau}^{||}\nabla_{\sigma}^{||}F_{\tau}|^2\, d\tau + 2|\nabla_{\sigma}^{||}F_{\sigma}|^2 \, d\delta_{\sigma}(\tau),
\end{align}
where $(\dot{g}+2\mathrm{Rc})_{\tau}(v,w) = (\dot{g_{t}}+2\mathrm{Rc}_{g_{t}})|_{t=T-\tau}(P_{\tau}^{-1}v,P_{\tau}^{-1}w)$ and $\dot{g} = \frac{d}{dt}g$.
\end{theorem}
\noindent This generalized Bochner formula proves to be a fundamental tool in characterizing the Ricci flow. Note that, if $(M,g_{t})_{t \in I}$ evolves by Ricci flow, this formula reduces to
\begin{align}
d(|\nabla_{\sigma}^{||}F_{\tau}|^2) &= \langle \nabla_{\tau}^{||}|\nabla_{\sigma}^{||}F_{\tau}|^2,dW_{\tau} \rangle + 2|\nabla_{\tau}^{||}\nabla_{\sigma}^{||}F_{\tau}|^2\, d\tau + 2|\nabla_{\sigma}^{||}F_{\sigma}|^2 \, d\delta_{\sigma}(\tau),
\end{align}
and this trivially implies the following infinite-dimensional generalization of Bochner inequality  \eqref{Bochner 2} in the time-dependent setting 
\begin{align} \label{Bochner gen fin}
d(|\nabla_{\sigma}^{||}F_{\tau}|^2) &\geq \langle \nabla_{\tau}^{||}|\nabla_{\sigma}^{||}F_{\tau}|^2,dW_{\tau} \rangle + 2|\nabla_{\tau}^{||}\nabla_{\sigma}^{||}F_{\tau}|^2\, d\tau + 2|\nabla_{\sigma}^{||}F_{\sigma}|^2 \, d\delta_{\sigma}(\tau).
\end{align}
In contrast to the heat flow Bochner inequality, this generalized martingale Bochner inequality \eqref{Bochner gen fin} as well as the estimates that follow from it are strong enough to help exhibit solutions and not just supersolutions of the Ricci flow.
\newline \\
Specifically, Theorem \ref{Bochner gen} has four main applications:

\begin{itemize}
\item a characterization of the Ricci flow via Bochner inequalities for martingales on parabolic path space;
\item gradient estimates for martingales on parabolic path space;
\item Hessian estimates for martingales on parabolic path space;
\item a new and much simpler proof of the characterization of solutions of the Ricci flow by Haslhofer and Naber in 2018 \cite[Theorem 1.22]{HN18a},
\end{itemize}
which will be discussed in Section \ref{applications}.
\newline \\
To explain the meaning of Theorem \ref{Bochner gen} in the simplest example, this generalized Bochner formula on $P\mathcal{M}$ directly reduces to the standard Bochner formula in the case of $1$-point functions, i.e. when $F_{\tau}(\gamma)$ satisfies equation \eqref{one-point heat flow}. That is, the evolution of $|\nabla H_{T-\tau_{1},T-\tau}f|^2$ for $\tau \leq \tau_{1}$ is calculated as
\begin{align}
&\left (-\partial_{\tau} - \Delta_{g_{T-\tau}} \right )|\nabla H_{T-\tau_{1},T-\tau} f|^{2} \leq -2|\nabla^{2} H_{T-\tau_{1},T-\tau}f|^2
\end{align}
in Corollary \ref{3.5}. Setting $s=T-\tau_{1}$ and $t=T-\tau$, this explicitly recovers \eqref{Bochner 2} from Section \ref{Background on Characterizations of Ricci Flow}.

\subsection{Applications} \label{applications}
We will conclude with some main applications of our Bochner inequality \eqref{Bochner gen fin}. First, we shall develop a new characterization of the Ricci flow.

\begin{theorem} (New characterizations of the Ricci Flow) \label{theorem char}
For an evolving family of manifolds $(M^{n},g_{t})_{t \in I}$, the following are equivalent to solving the Ricci flow $\partial_{t}g_{t}=-2\mathrm{Rc}_{g_{t}}$:

\begin{description}[font=\mdseries]
    \item[(C1)] Martingales on parabolic path space satisfy the full Bochner inequality
    \begin{align}
    d|\nabla_{\sigma}^{||}F_{\tau}|^{2} \geq \langle \nabla_{\tau}|\nabla_{\sigma}^{||}F_{\tau}|^{2},dW_{\tau} \rangle + 2|\nabla_{\tau}^{||}\nabla_{\sigma}^{||}F_{\tau}|^{2}\, d\tau + 2|\nabla_{\sigma}^{||}F_{\sigma}|^{2}\, d\delta_{\sigma}(\tau)
    \end{align}
    \item[(C2)] Martingales on parabolic path space satisfy the dimensional Bochner inequality
    \begin{align}
     d|\nabla_{\sigma}^{||}F_{\tau}|^{2} \geq \langle \nabla_{\tau}|\nabla_{\sigma}^{||}F_{\tau}|^{2},dW_{\tau} \rangle + \frac{2}{n}|\Delta_{\sigma,\tau}^{||}F_{\tau}|^{2}\, d\tau + 2|\nabla_{\sigma}^{||}F_{\sigma}|^{2}\, d\delta_{\sigma}(\tau)
    \end{align}
    \item[(C3)] Martingales on parabolic path space satisfy the weak Bochner inequality
    \begin{align}
    d|\nabla_{\sigma}^{||}F_{\tau}|^{2} \geq \langle \nabla_{\tau}|\nabla_{\sigma}^{||}F_{\tau}|^{2},dW_{\tau} \rangle + 2|\nabla_{\sigma}^{||}F_{\sigma}|^{2} \, d\delta_{\sigma}(\tau)
    \end{align}
    \item[(C4)] Martingales on parabolic path space satisfy the linear Bochner inequality
    \begin{align}
    d|\nabla_{\sigma}^{||}F_{\tau}| \geq \langle \nabla_{\tau}|\nabla_{\sigma}^{||}F_{\tau}|,dW_{\tau} \rangle + |\nabla_{\sigma}^{||}F_{\sigma}| \, d\delta_{\sigma}(\tau)
    \end{align}
    \item[(C5)] If $F_{\tau}$ is a martingale, then $\tau \rightarrow |\nabla_{\sigma}^{||}F_{\tau}|$ is a submartingale for every $\sigma \geq 0$.
\end{description}
\end{theorem}

\noindent Second, we shall obtain gradient estimates for martingales on parabolic path space.

\begin{theorem} (Gradient Estimates for Martingales on Parabolic Path Space) \label{theorem grad}
For an evolving family of manifolds $(M^{n},g_{t})_{t \in I}$, the following are equivalent to solving the Ricci flow $\partial_{t}g_{t}=-2\mathrm{Rc}_{g_{t}}$:

\begin{description}[font=\mdseries]
\item[(G1)] For any $F \in L^{2}(P\mathcal{M})$, $\sigma$ fixed and $\tau_{1}\leq \tau_{2}$, the induced martingale satisfies the gradient estimate
\begin{align}
|\nabla_{\sigma}^{||}F_{\tau_{1}}| \leq \mathbb{E}_{(x,T)} \left [|\nabla_{\sigma}^{||}F_{\tau_{2}}| \big | \Sigma_{\tau_{1}} \right ].
\end{align}
\item[(G2)] For any $F \in L^{2}(P\mathcal{M})$, $\sigma$ fixed and $\tau_{1}\leq \tau_{2}$, the induced martingale satisfies the gradient estimate
\begin{align}
|\nabla_{\sigma}^{||}F_{\tau_{1}}|^{2} \leq \mathbb{E}_{(x,T)} \left [|\nabla_{\sigma}^{||}F_{\tau_{2}}|^{2} \big | \Sigma_{\tau_{1}} \right ].
\end{align}
\end{description}
\end{theorem}
\noindent Note that in the case of $\sigma=\tau_{1}=0$, $(\mathrm{G1})$ reduces to the infinite-dimensional gradient estimate \eqref{gradient 2 inf}.
\newline \\
\noindent Next, we shall obtain Hessian estimates for martingales on parabolic path space.

\begin{theorem} (Hessian Estimates for Martingales on Parabolic Path Space) \label{theorem hess}
For an evolving family of manifolds $(M^{n},g_{t})_{t \in I}$ that solve the Ricci flow $\partial_{t}g_{t} = -2\mathrm{Rc}_{g_{t}}$ and a function $F \in L^{2}(P\mathcal{M})$, it holds that:

\begin{description}[font=\mdseries]
\item[(H1)] For each $\sigma \geq 0$, we have the estimate
\begin{align}
\mathbb{E}_{(x,T)} \left [|\nabla_{\sigma}^{||}F_{\sigma}|^{2} \right ] &+ 2\mathbb{E}_{(x,T)} \int_{0}^{T} \left [|\nabla_{\tau}^{||}\nabla_{\sigma}^{||}F_{\tau}|^{2} \right ]\, d\tau \leq \mathbb{E}_{(x,T)} \left [|\nabla_{\sigma}^{||}F|^{2} \right ].
\end{align}
\item[(H2)] We have the Poincar\'e Hessian estimate
\begin{align}
&\mathbb{E}_{(x,T)} \left [\left (F - \mathbb{E}_{(x,T)}[F] \right )^{2}\right ] \nonumber \\
&\quad + 2\int_{0}^{T} \int_{0}^{T} \mathbb{E}_{(x,T)} \left [|\nabla_{\tau}^{||}\nabla_{\sigma}^{||}F_{\tau}|^{2} \right ]\, d\sigma\, d\tau \leq \int_{0}^{T} \mathbb{E}_{(x,T)} \left [|\nabla_{\sigma}^{||}F|^{2} \right ]\, d\sigma.
\end{align}
\item[(H3)] We have the log-Sobolev Hessian estimate
\begin{align}
&\mathbb{E}_{(x,T)} \left [F^{2}\ln(F^{2})\right ] - \mathbb{E}_{(x,T)}[F^{2}] \ln \left ( \mathbb{E}_{(x,T)}[F^{2}] \right ) \\
&\quad + 2 \int_{0}^{T} \int_{0}^{T} \mathbb{E}_{(x,T)} \left [(F^{2})_{\tau} |\nabla_{\tau}^{||}\nabla_{\sigma}^{||} \ln((F^{2})_{\tau})|^{2} \right ]\, d\sigma\, d\tau \leq 4 \int_{0}^{T} \mathbb{E}_{(x,T)} \left [|\nabla_{\sigma}^{||}F|^{2} \right ]\, d\sigma. \nonumber
\end{align}
\end{description}
\end{theorem}

\noindent Finally, our generalized Bochner formula on parabolic path space leads to a simpler proof of the characterization of solutions of the Ricci flow found by Haslhofer and Naber \cite{HN18a}.

\begin{theorem} (Characterization of Solutions of the Ricci Flow) \cite[Theorem 1.22]{HN18a} \label{theorem Ricci}
For an evolving family of manifolds $(M^{n},g_{t})_{t \in I}$, the following are equivalent:

\begin{description}[font=\mdseries]
\item[(R1)] $(M^{n},g_{t})_{t \in I}$ solves the Ricci flow $\partial_{t}g_{t}=-2\mathrm{Rc}_{g_{t}}$.
\item[(R2)] For every $F \in L^{2}(P\mathcal{M})$, we have the gradient estimate
\begin{align}
\left |\nabla_{x} \mathbb{E}_{(x,T)}[F] \right | \leq \mathbb{E}_{(x,T)}[|\nabla_{0}^{||}F|].
\end{align}
\item[(R3)] For every $F \in L^{2}(P\mathcal{M})$, the induced martingale $\{F_{\tau}\}_{\tau \in [0,T]}$ satisfies the quadratic variation estimate
\begin{align}
\mathbb{E}_{(x,T)} \left [\frac{d[F,F]_{\tau}}{d\tau} \right ] \leq 2\mathbb{E}_{(x,T)} \left [|\nabla_{\tau}^{||}F|^{2} \right ].
\end{align}
\item[(R4)] The Ornstein-Uhlenbeck operator $\mathcal{L}_{(\tau_{1},\tau_{2})}$ on parabolic path space $L^2(P\mathcal{M})$ satisfies the log-Sobolev inequality
\begin{align}
\mathbb{E}_{(x,T)} \left [(F^{2})_{\tau_{2}} \log((F^{2})_{\tau_{2}}) - (F^{2})_{\tau_{1}} \log((F^{2})_{\tau_{1}}) \right ] \leq 2 \mathbb{E}_{(x,T)} \left [\langle F, \mathcal{L}_{(\tau_{1},\tau_{2})}F \rangle_{\mathcal{H}} \right ].
\end{align}
\item[(R5)] The Ornstein-Uhlenbeck operator $\mathcal{L}_{(\tau_{1},\tau_{2})}$ on parabolic path space $L^2(P\mathcal{M})$ satisfies the Poincar\'e inequality
\begin{align}
\mathbb{E}_{(x,T)} \left [(F_{\tau_{2}}-F_{\tau_{1}})^{2} \right ] \leq \mathbb{E}_{(x,T)} \left [\langle F, \mathcal{L}_{(\tau_{1},\tau_{2})}F \rangle_{\mathcal{H}} \right ].
\end{align}
\end{description}
\end{theorem}
\noindent Our new proof is much shorter. For example, to derive $(\mathrm{R2})$, integrate $(\mathrm{C4})$ from $0$ to $T$, and take expectations 
\begin{align}
&\mathbb{E}_{(x,T)} \left [\int_{0}^{T} d|\nabla_{\sigma}^{||}F_{\tau}|\, d\tau \right ] \overset{(\mathrm{C4})}{\geq} \mathbb{E}_{(x,T)} \left [\int_{0}^{T} \langle \nabla_{\tau}|\nabla_{\sigma}^{||}F_{\tau}|, dW_{\tau} \rangle + |\nabla_{\sigma}^{||}F_{\sigma}|d\delta_{\sigma}(\tau) \right ] \nonumber \\
&\quad \implies \mathbb{E}_{(x,T)} \left [|\nabla_{\sigma}^{||}F| \right ] - \mathbb{E}_{(x,T)} \left [|\nabla_{\sigma}^{||}F_{\sigma}| \right ] \geq 0
\end{align}
Then take limits as $\sigma \rightarrow 0$ to yield the result
\begin{align}
|\nabla_{x} \mathbb{E}_{(x,T)}[F]| = \mathbb{E}_{(x,T)} \left [|\nabla_{0}^{||}F_{0}| \right ] \leq \mathbb{E}_{(x,T)} \left [|\nabla_{0}^{||}F| \right ].
\end{align}
The article is organized as follows:
\begin{itemize}
\item In Section 2, we shall discuss the geometric and probabilistic preliminaries needed for the proofs ouf our main theorems.
\item In Section 3, we shall prove Theorem \ref{Bochner gen}, the Bochner formula for martingales on parabolic path space. %in particular the generalized Bochner formula of \eqref{Bochner 2} that analogizes \eqref{Bochner inf}.
\item In Section 4, we shall discuss the four aforementioned applications of our analysis on path space, i.e. Theorems \ref{theorem char}, \ref{theorem grad}, \ref{theorem hess} and \ref{theorem Ricci}.
\end{itemize}
\textbf{Acknowledgements.} The author has been supported by the Ontario Graduate Scholarship and he acknowledges his supervisor Robert Haslhofer for his invaluable guidance and support in bringing this paper into fruition.

\pagebreak

\section{Preliminaries}

\subsection{Geometric Preliminaries}
To begin this section, we shall recall the concepts relevant to the construction of the frame bundle on evolving manifolds. An expression of the canonical horizontal ($H_{a}$ and $D_{t}$) and vertical ($V_{ab}$) vector fields and their commutators will complete this preliminary section.
\newline \\
In time-independent geometry, given a complete Riemannian manifold $M$, one considers the orthonormal frame bundle $\pi : F \rightarrow M$, where the fibres are orthonormal maps  $F_{x} := \left \{u: \R^{n} \rightarrow T_{x}M \, \text{orthonormal} \right \}$. To each curve $x_{t} \in M$, one can associate a horizontal lift $u_{t} \in F$. In particular, to each vector $X \in T_{x}M$, given $u \in \pi^{-1}(x)$, one can associate its horizontal lift $X^{*} \in T_{u}F$.
\newline \\
We shall now explain, following  $\cite{Ham93}$ and $\cite{HN18a}$, how these notions can be adapted to the time-dependent setting.
\noindent To make the appropriate adjustment, we begin by defining space-time $\mathcal{M}$ and the equipped connection $\nabla$ as follows:
\begin{definition}\label{1} (Space-time)
Let $(M,g_{t})_{t \in I}$ be an evolving family of Riemannian manifolds. The space-time is then defined as $\mathcal{M} = M \times I$ equipped with the space-time connection defined on vector fields by $\nabla_{X}Y = \nabla_{X}^{g_{t}}Y$ and $\nabla_{t}Y = \partial_{t}Y + \frac{1}{2} \partial_{t}g_{t}(Y, \cdot)^{\#g_{t}}$.
\end{definition}
\noindent Also observe that this choice of connection is compatible with the metric, namely
\begin{align}
\frac{d}{dt} \langle X, Y \rangle _{g_{t}} = \langle \nabla_{t}X, Y \rangle _{g_{t}} + \langle X, \nabla_{t}Y \rangle _{g_{t}}.
\end{align}
Generalizing the earlier time-independent construction, we consider the $\mathcal{O}_{n}$-bundle $\pi: \mathcal{F} \rightarrow \mathcal{M}$, where the fibres are given by $\mathcal{F}_{(x,t)} := \left \{u: \R^{n} \rightarrow (T_{x}M, g_{t}) \, \text{orthonormal} \right \}$.
\newline \\
To each curve $\gamma_{t} \in \mathcal{M}$, we can now associate a horizontal lift $u_{t} \in \mathcal{F}$. Namely, given $u_{0} \in \pi^{-1}(\gamma_{0})$, the curve $u_{t}$ is the unique solution of $\pi(u_{t})=\gamma_{t}$ and $\nabla_{\dot{\gamma_{t}}}(u_{t}e_{a})=0$ for $a \in \left \{1, 2,... ,n \right \}$, where $\nabla$ is the space-time connection from Definition \ref{1}. More explicitly, we provide the following formal definition:
\begin{definition}\label{2} (Horizontal lift)
Given a vector $\alpha X + \beta \partial_{t} \in T_{(x,t)}\mathcal{M}$ and a frame $u \in \mathcal{F}_{(x,t)}$, there is a unique horizontal lift $\alpha X^{*} + \beta D_{t}$ satisfying $\pi_{*}(\alpha X^{*} + \beta D_{t}) = \alpha X + \beta \partial_{t}$. In particular, $X^{*}$ is the horizontal lift of $X \in T_{x}M$ with respect to the fixed metric $g_{t}$.
\end{definition}
\noindent Note that there are $n+1$ canonical horizontal vector fields on $\mathcal{F}$, namely the time-like horizontal vector field $D_{t}$ defined as the horizontal lift of $\partial_{t}$ and the space-like horizontal vector fields $\left \{H_{a} \right \}_{a=1}^{n}$ defined by $H_{a}(u) = (ue_{a})^{*}$.  Also note the notion of vertical vector fields given by $V_{ab}(u) = \frac{d}{d\varepsilon}|_{\varepsilon=0}(u\exp(\varepsilon A_{ab}))$ where $(A_{ab})_{cd} = (\delta_{ac}\delta_{bd}-\delta_{bc}\delta_{ad}) \in M_{n}(\R)$. We now want to express these horizontal and vertical vector fields in local coordinates as follows:
\begin{definition} (Local coordinates)
We view $\mathcal{F}$ as a sub-bundle of the $GL_{n}$-bundle $\pi: \mathcal{G} \rightarrow \mathcal{M}$ where $\mathcal{G}_{(x,t)} := \left \{ u: \R^{n} \rightarrow (T_{x}M, g_{t}) \, \text{invertible, linear} \right \}$. Then, when given local coordinates $(x^{1}, ..., x^{n},t)$ on $\mathcal{M}$, we get local coordinates $(x^{i}, t, e_{a}^{j})$ on $\mathcal{G}$, where $e_{a}^{j}$ is defined by $ue_{a} = e_{a}^{j} \frac{\partial}{\partial x^{j}}$.
\end{definition}
\noindent Also note that on $\mathcal{F}$ we have $\delta_{ab}=g(ue_{a},ue_{b})=g_{ij}e_{a}^{i}e_{b}^{j}$ and thus we can express the inverse metric as
\begin{align}
g^{ij}=e_{a}^{i}e_{a}^{j}.
\end{align}
It now remains in this section to both write out the canonical vector fields explicitly in local coordinates and derive some commutator relations between them.
\begin{lemma} [cf. \cite{Ham93}]
In local coordinates, the canonical horizontal vector fields $H_{a}$ and $D_{t}$ and canonical vertical vector fields $V_{ab}$ can be expressed as
\begin{align}
\begin{cases}
H_{a} &= e_{a}^{j} \frac{\partial}{\partial x^{j}} - e_{a}^{j} e_{b}^{k} \Gamma_{jk}^{\l} \frac{\partial}{\partial e_{b}^{\l}} \\
V_{ab} &= e_{b}^{j}\frac{\partial}{\partial e_{a}^{j}} - e_{a}^{j} \frac{\partial}{\partial e_{b}^{j}} \\
D_{t} &= \partial_{t} - \frac{1}{2} \widetilde{\partial_{t}g}_{ab} e_{b}^{\l} \frac{\partial}{\partial e_{a}^{\l}} \label{4},
\end{cases}
\end{align}
where $(\widetilde{\partial_{t}g})_{ab}(u) := (\partial_{t}g)_{\pi(u)}(ue_{a},ue_{b})$.
\end{lemma}
\begin{proof}
The canonical horizontal vector fields, $H_{a}$ are exactly the same as in \cite{Ham93}. 
\newline \\
Next, considering the curve $u(\varepsilon)=u\exp(\varepsilon A_{ab})$, recall that $e_{c}^{j}$ and $A_{ab}e_{c}$ are defined via the relations $ue_{c} = e_{c}^{j} \frac{\partial}{\partial x^{j}}$  and $A_{ab}e_{c} = \delta_{ca}e_{b} - \delta_{cb}e_{a}$. Then derive
\begin{align}
\dot{u}(0)e_{c} = \dot{e}_{c}^{j}(0) \frac{\partial}{\partial x^{j}} = uA_{ab}e_{c} = \delta_{ac}ue_{b} - \delta_{bc}ue_{a} = \left (\delta_{ac}e_{b}^{j} - \delta_{bc}e_{a}^{j} \right )\frac{\partial}{\partial x^{j}},
\end{align}
whence
\begin{align}
V_{ab} = \dot{u}(0) = \dot{e}_{c}^{j}(0) \frac{\partial}{\partial e_{c}^{j}} = \left (\delta_{ac}e_{b}^{j} - \delta_{bc}e_{a}^{j} \right )\frac{\partial}{\partial e_{c}^{j}} = e_{b}^{j} \frac{\partial}{\partial e_{a}^{j}} - e_{a}^{j} \frac{\partial}{\partial e_{b}^{j}}.
\end{align}
\newline \\
Finally we recall that $D_{t}$ is defined as the horizontal lift of $\partial_{t}$. More explicitly, given $u_{0} \in \mathcal{F}$, suppose $\pi(u_{0}) = (x_{0},t_{0})$ and $\gamma_{t} := (x_{0},t_{0}+t)$ and let $u_{t}$ be the horizontal lift of $\gamma_{t}$. Then, we have that $D_{t}(u_{0}) = \frac{d}{dt}|_{t=0} u_{t}$. Recalling Definition \ref{1},  and using the tensorial transformation rule $\widetilde{(\partial_{t}g)}_{ab} = \partial_{t}g_{jk}e_{a}^{j}e_{b}^{k}$ (see equation \eqref{14} below), we compute 
\begin{align}
\nabla_{t} \left (e_{a}^{j}\frac{\partial}{\partial x^{j}} \right ) &= \frac{d(e_{a}^{j})}{dt} \frac{\partial}{\partial x^{j}} + e_{a}^{j} \, \nabla_{t} \left (\frac{\partial}{\partial x^{j}} \right ) \nonumber \\
&=  \frac{d(e_{a}^{\l})}{dt} \frac{\partial}{\partial x^{\l}} + \frac{1}{2}e_{a}^{j}\partial_{t}g_{jk}g^{k\l}\frac{\partial}{\partial x^{\l}} \nonumber \\
&= \frac{d(e_{a}^{\l})}{dt} \frac{\partial}{\partial x^{\l}} + \frac{1}{2}e_{a}^{j}\partial_{t}g_{jk}e_{b}^{k}e_{b}^{\l}\frac{\partial}{\partial x^{\l}} \nonumber \\
&= \frac{d(e_{a}^{\l})}{dt} \frac{\partial}{\partial x^{\l}}  + \frac{1}{2} \widetilde{\partial_{t}g}_{ab}e_{b}^{\l} \frac{\partial}{\partial x^{\l}}.
\end{align}
It follows that, since $\dot{\gamma}_{t} = \partial_{t}$ and $u_{t}e_{a}=e_{a}^{\l}(t) \frac{\partial}{\partial x^{\l}}$,
\begin{align}
\nabla_{\dot{\gamma}_{t}}(u_{t}e_{a}) = \nabla_{t} \left (e_{a}^{\l}(t) \frac{\partial}{\partial x^{\l}} \right ) = \left (\frac{d}{dt}(e_{a}^{\l}(t)) + \frac{1}{2}\widetilde{\partial_{t}g}_{ab}e_{b}^{\l} \right )\frac{\partial}{\partial x^{\l}} = 0.
\end{align}
By exhibiting $D_{t}(u_{0})$ in local coordinates
\begin{align}
D_{t}(u_{0}) = 0 \cdot \frac{\partial}{\partial x^{j}} + 1 \cdot \frac{\partial}{\partial t} + \frac{d}{dt}|_{t=0}(e_{a}^{\l}(t)) \frac{\partial}{\partial e_{a}^{\l}} = \partial_{t} - \frac{1}{2}\widetilde{\partial_{t}g}_{ab}e_{b}^{\l} \frac{\partial}{\partial e_{a}^{\l}},
\end{align}
we conclude the proof.
\end{proof}
\noindent We now recall that the time-dependent tensor fields $T$ correspond to equivariant functions $\tilde{T}$ on $\mathcal{F}$. For example, a function $f: \mathcal{M} \rightarrow \R$ corresponds to the invariant function $\tilde{f} = f \circ \pi: \mathcal{F} \rightarrow \R$ and a time-dependent two-tensor $T = T_{ij}(x,t) \, dx^{i} \otimes dx^{j}$ corresponds to an equivariant function $\tilde{T}=(\tilde{T}_{ab}): \mathcal{F} \rightarrow \R^{n \times n}$ via $\tilde{T}_{ab}(u)=T_{\pi(u)}(ue_{a},ue_{b})$. Note that identities $ue_{a} = e_{a}^{j} \frac{\partial}{\partial x^{j}}$ and $ue_{b} = e_{b}^{k} \frac{\partial}{\partial x^{k}}$ yield the transformation rule
\begin{align}
\tilde{T}_{ab} = T_{ij}e_{a}^{i}e_{b}^{j} \label{14}.
\end{align}
Also observe that using equations \eqref{4} and \eqref{14}, one obtains the formula
\begin{align}
V_{ab}\tilde{T}_{cd} %&= \left (e_{b}^{k} \frac{\partial}{\partial e_{a}^{k}} - e_{a}^{k} \frac{\partial}{\partial e_{b}^{k}} \right )(T_{ij}e_{c}^{i}e_{d}^{j}) \\
%&= T_{ij} \left (e_{b}^{k}\delta_{c}^{a}\delta_{k}^{i}e_{d}^{j} + e_{b}^{k}e_{c}^{i}\delta_{d}^{a}\delta_{k}^{j}e_{c}^{i} - e_{a}^{k}\delta_{c}^{b}\delta_{k}^{i}e_{d}^{j} - e_{a}^{k}e_{c}^{i}\delta_{d}^{b}\delta_{k}^{j} \right ) \\
%&= T_{ij} \left (e_{b}^{i}e_{d}^{j}\delta_{c}^{a} + e_{b}^{j}e_{c}^{i}\delta_{d}^{a} - e_{a}^{i}e_{d}^{j}\delta_{c}^{b} - e_{a}^{j}e_{c}^{i}\delta_{d}^{b} \right ) \\
&= \tilde{T}_{bd}\delta_{c}^{a} - \tilde{T}_{ad}\delta_{c}^{b} + \tilde{T}_{cb}\delta_{d}^{a} - \tilde{T}_{ca}\delta_{d}^{b} \label{17}.
\end{align}
\begin{proposition} (Derivatives) \cite{HN18a} \label{derivatives}
From the correspondence with equivariant functions, the first and second order derivatives of tensor fields can be computed as follows
\begin{align}
\begin{cases}
&\widetilde{\nabla_{X}T} = X^{*}\tilde{T} \\
&\widetilde{\nabla_{t}T} = D_{t}\tilde{T} \\
&\widetilde{\Delta T} = \sum_{a=1}^{n} H_{a}H_{a}\tilde{T} =: \Delta_{H}\tilde{T} \\
&(\nabla^{2}f)(ue_{a},ue_{b}) = H_{a}H_{b}\tilde{f}
\end{cases}
\end{align}
\end{proposition}
\begin{proof}
Except for the fourth identity regarding the Hessian, these are either classical results from differential geometry or have already been proven in Lemmas 3.1 and 3.3 of \cite{HN18a}. For this last identity, write the canonical horizontal vector fields in local coordinates and compute
\begin{align}
H_{a}H_{b}\tilde{f} &= \left (e_{a}^{j} \frac{\partial}{\partial x^{j}} - e_{a}^{j}e_{c}^{k}\Gamma_{jk}^{\l} \frac{\partial}{\partial e_{c}^{\l}} \right )e_{b}^{p} \frac{\partial}{\partial x^{p}}\tilde{f} \nonumber \\
&= e_{a}^{j}e_{b}^{k} \left (\frac{\partial}{\partial x^{j}} \frac{\partial}{\partial x^{k}} \tilde{f} - \Gamma_{jk}^{p} \frac{\partial}{\partial x^{p}}\tilde{f} \right ) \nonumber \\
&= e_{a}^{j}e_{b}^{k}\nabla_{j}\nabla_{k}f \nonumber \\
&= \nabla^{2}f(ue_{a},ue_{b}),
\end{align}
thereby proving the proposition.
\end{proof}
\noindent Next we proceed to prove a few commutator relations between the newly defined vector field, $D_{t}$, and the canonical horizontal and vertical vector fields.
\begin{lemma}[cf. \cite{Ham93}] \label{1.6} The fundamental vectors fields on the frame bundle satisfy the following commutator relations
\begin{align}
\begin{cases}
[H_{a}, H_{b}] &= \frac{1}{2}R_{abcd}V_{cd} \\
[V_{ab},H_{c}] &= \delta_{ac}H_{b}-\delta_{bc}H_{a} \\
[V_{ab},V_{cd}] &= \delta_{bd}V_{ac} - \delta_{ad}V_{bc} + \delta_{ac}V_{bd} - \delta_{bc}V_{ad} \\
[D_{t},H_{a}] &= -\frac{1}{2}\widetilde{\partial_{t}g}_{ad}H_{d} + \frac{1}{2}H_{b}\widetilde{\partial_{t}g}_{ac}V_{cb} \\
[D_{t},V_{ab}] &= 0.
\end{cases}
\end{align}
\end{lemma}
\begin{proof}
We admit without proof the commutator relations not involving $D_{t}$ as they can be derived from basic differential geometry. It remains to check the final two commutator relations. To prove the first of these, between $D_{t}$ and the horizontal vector field $H_{a}$, we compute
\begin{align}
[\partial_{t}, H_{a}] &= \left [\partial_{t}, e_{a}^{j}\frac{\partial}{\partial x^{j}} - e_{a}^{j}e_{b}^{k}\Gamma_{jk}^{\l}\frac{\partial}{\partial e_{b}^{\l}} \right ] \nonumber \\
&= -e_{a}^{j}e_{b}^{k} \partial_{t}\Gamma_{jk}^{\l}\frac{\partial}{\partial e_{b}^{\l}} \nonumber \\
&= -\frac{1}{2}e_{a}^{j}e_{b}^{k} (g^{\l p}(\nabla_{j}(\partial_{t}g_{kp}) + \nabla_{k}(\partial_{t}g_{jp}) - \nabla_{p}(\partial_{t}g_{jk}))\frac{\partial}{\partial e_{b}^{\l}} \nonumber \\
&= -\frac{1}{2}e_{a}^{j}e_{b}^{k}e_{c}^{\l}e_{c}^{p} \left (\nabla_{j}(\partial_{t}g_{kp}) + \nabla_{k}(\partial_{t}g_{jp}) - \nabla_{p}(\partial_{t}g_{jk}) \right ) \frac{\partial}{\partial e_{b}^{\l}} \nonumber \\
&= -\frac{1}{2}e_{c}^{\l} \left ((\widetilde{\nabla \partial_{t} g})_{abc} + (\widetilde{\nabla \partial_{t} g})_{bac} - (\widetilde{\nabla \partial_{t} g})_{cab} \right ) \frac{\partial}{\partial e_{b}^{\l}} \nonumber \\
&= -\frac{1}{2} e_{c}^{\l} (H_{a} \widetilde{\partial_{t}g}_{bc} + H_{b} \widetilde{\partial_{t}g}_{ac} - H_{c} \widetilde{\partial_{t}g}_{ab}) \frac{\partial}{\partial e_{b}^{\l}} \nonumber \\
%&= -\frac{1}{2} H_{a}(\widetilde{\partial_{t}g}_{bc})e_{c}^{\l} \frac{\partial}{\partial e_{b}^{\l}} - \frac{1}{2}H_{b}(\widetilde{\partial_{t}g}_{ac})e_{c}^{\l} \frac{\partial}{\partial e_{b}^{\l}} + \frac{1}{2}H_{b}(\widetilde{\partial_{t}g}_{ac})e_{b}^{\l} \frac{\partial}{\partial e_{c}^{\l}} \\
&= -\frac{1}{2} H_{a}(\widetilde{\partial_{t}g}_{bc})e_{c}^{\l} \frac{\partial}{\partial e_{b}^{\l}} + \frac{1}{2}H_{b}(\widetilde{\partial_{t}g}_{ac}) \left (e_{b}^{\l} \frac{\partial}{\partial e_{c}^{\l}} - e_{c}^{\l} \frac{\partial}{\partial e_{b}^{\l}} \right ) \nonumber \\
&= -\frac{1}{2} H_{a}(\widetilde{\partial_{t}g}_{bc})e_{c}^{\l} \frac{\partial}{\partial e_{b}^{\l}} + \frac{1}{2}H_{b}(\widetilde{\partial_{t}g}_{ac})V_{cb} \label{31}
\end{align}
and
\begin{align}
[D_{t}-\partial_{t}, H_{a}] &= [-\frac{1}{2}(\widetilde{\partial_{t}g}_{cd})e_{d}^{\l'}\frac{\partial}{\partial e_{c}^{\l'}}, H_{a}] \nonumber \\
&= -\frac{1}{2}\widetilde{\partial_{t}g}_{cd} \left [e_{d}^{\l'}\frac{\partial}{\partial e_{c}^{\l'}}, H_{a} \right ] + \frac{1}{2}H_{a}(\widetilde{\partial_{t}g}_{cd})e_{d}^{\l'}\frac{\partial}{\partial e_{c}^{\l'}} \nonumber \\
%&= -\frac{1}{2}\widetilde{\partial_{t}g}_{cd} \left (e_{d}^{\l'} \delta_{a}^{c}\delta_{\l'}^{j}\frac{\partial}{\partial x^{j}} - e_{d}^{\l'} (\delta_{a}^{c}\delta_{\l'}^{j}e_{b}^{k} + e_{a}^{j}\delta_{b}^{c}\delta_{\l'}^{k})\Gamma_{jk}^{\l}\frac{\partial}{\partial e_{b}^{\l}} \right ) \\
%&\quad -\frac{1}{2}\widetilde{\partial_{t}g}_{cd}e_{a}^{j}e_{b}^{k}\Gamma_{jk}^{l}\delta_{d}^{b}\delta_{\l}^{\l'}\frac{\partial}{\partial e_{c}^{\l'}} + \frac{1}{2}H_{a}(\widetilde{\partial_{t}g}_{bc})e_{c}^{\l}\frac{\partial}{\partial e_{b}^{\l}} \\
%&= -\frac{1}{2}\widetilde{\partial_{t}g}_{cd} \left (\delta_{a}^{c} e_{d}^{j} \frac{\partial}{\partial x^{j}} - (\delta_{a}^{c}e_{d}^{j}e_{b}^{k} + \delta_{b}^{c}e_{a}^{j}e_{d}^{k})\Gamma_{jk}^{\l}\frac{\partial}{\partial e_{b}^{\l}} \right ) \\
%&\quad -\frac{1}{2}\widetilde{\partial_{t}g}_{cd}\delta_{d}^{b}e_{a}^{j}e_{b}^{k}\Gamma_{jk}^{\l}\frac{\partial}{\partial e_{c}^{\l}} + \frac{1}{2}H_{a}(\widetilde{\partial_{t}g}_{bc})e_{c}^{\l}\frac{\partial}{\partial e_{b}^{\l}} \\
%&= -\frac{1}{2}\widetilde{\partial_{t}g}_{ad}\left (e_{d}^{j} \frac{\partial}{\partial x^{j}} - e_{d}^{j}e_{b}^{k}\Gamma_{jk}^{\l}\frac{\partial}{\partial e_{b}^{\l}} \right ) + \frac{1}{2}\widetilde{\partial_{t}g}_{bd}e_{a}^{j}e_{d}^{k}\Gamma_{jk}^{\l}\frac{\partial}{\partial e_{b}^{\l}} \\
%&\quad - \frac{1}{2}\widetilde{\partial_{t}g}_{cd}e_{a}^{j}e_{d}^{k}\Gamma_{jk}^{\l}\frac{\partial}{\partial e_{c}^{\l}} + \frac{1}{2}H_{a}(\widetilde{\partial_{t}g}_{bc})e_{c}^{\l}\frac{\partial}{\partial e_{b}^{\l}} \\
&= -\frac{1}{2}\widetilde{\partial_{t}g}_{cd}\delta_{a}^{c}H_{d} + \frac{1}{2}H_{a}(\widetilde{\partial_{t}g}_{bc})e_{c}^{\l}\frac{\partial}{\partial e_{b}^{\l}} \nonumber \\
&= -\frac{1}{2}\widetilde{\partial_{t}g}_{ad}H_{d} + \frac{1}{2}H_{a}(\widetilde{\partial_{t}g}_{bc})e_{c}^{\l}\frac{\partial}{\partial e_{b}^{\l}} \label{35}.
\end{align}
Next, we sum equations \eqref{31} and \eqref{35} to compute the desired commutator relation
\begin{align}
[D_{t},H_{a}] &= [\partial_{t}, H_{a}] + [D_{t}-\partial_{t},H_{a}] = -\frac{1}{2}\widetilde{\partial_{t}g}_{ad}H_{d} + \frac{1}{2}H_{b}(\widetilde{\partial_{t}g}_{ac})V_{cb}.
\end{align}
Finally, using equations \eqref{4} and \eqref{17}, the commutator of $D_{t}$ and $V_{ab}$ is
\begin{align}
[D_{t},V_{ab}] &= \left [\partial_{t}-\frac{1}{2}(\widetilde{\partial_{t}g}_{cd})e_{d}^{\l'}\frac{\partial}{\partial e_{c}^{\l'}}, V_{ab} \right ] \nonumber \\
&= -\frac{1}{2} \widetilde{\partial_{t}g}_{cd}\left [e_{d}^{\l'}\frac{\partial}{\partial e_{c}^{\l'}}, V_{ab} \right ] + \frac{1}{2}V_{ab}(\widetilde{\partial_{t}g}_{cd})e_{d}^{\l'}\frac{\partial}{\partial e_{c}^{\l'}} \nonumber \\
&= \frac{1}{2} \widetilde{\partial_{t}g}_{cd} \left (e_{d}^{\l'}\frac{\partial}{\partial e_{b}^{\l'}}\delta_{c}^{a} + e_{b}^{\l'}\frac{\partial}{\partial e_{c}^{\l'}}\delta_{d}^{a} - e_{d}^{\l'}\frac{\partial}{\partial e_{a}^{\l'}}\delta_{c}^{b} - e_{a}^{\l'}\frac{\partial}{\partial e_{c}^{\l'}}\delta_{d}^{b} \right ) \nonumber \\
&\quad +\frac{1}{2}\left (\widetilde{\partial_{t}g}_{bd}\delta_{c}^{a} + \widetilde{\partial_{t}g}_{cb}\delta_{d}^{a} - \widetilde{\partial_{t}g}_{ad}\delta_{c}^{b} - \widetilde{\partial_{t}g}_{ca}\delta_{d}^{b} \right )e_{d}^{\l'}\frac{\partial}{\partial e_{c}^{\l'}} \nonumber \\
%&= \frac{1}{2}\widetilde{\partial_{t}g}_{ad} e_{d}^{\l'}\frac{\partial}{\partial e_{b}^{\l'}} + \frac{1}{2}\widetilde{\partial_{t}g}_{ca}e_{b}^{\l'}\frac{\partial}{\partial e_{c}^{\l'}} - \frac{1}{2}\widetilde{\partial_{t}g}_{bd}e_{d}^{\l'}\frac{\partial}{\partial e_{a}^{\l'}} - \frac{1}{2}\widetilde{\partial_{t}g}_{cb}e_{a}^{\l'}\frac{\partial}{\partial e_{c}^{\l'}} \\
%&\quad + \frac{1}{2}\widetilde{\partial_{t}g}_{bd}e_{d}^{\l'}\frac{\partial}{\partial e_{a}^{\l'}} + \frac{1}{2}\widetilde{\partial_{t}g}_{cb}e_{a}^{\l'}\frac{\partial}{\partial e_{c}^{\l'}} - \frac{1}{2}\widetilde{\partial_{t}g}_{ad}e_{d}^{\l'}\frac{\partial}{\partial e_{b}^{\l'}} - \frac{1}{2}\widetilde{\partial_{t}g}_{ca}e_{b}^{\l'}\frac{\partial}{\partial e_{c}^{\l'}} \\
&\equiv 0,
\end{align}
thereby proving this lemma on commuting canonical vector fields.
\end{proof}
\begin{corollary} \label{1.7}
If $\tilde{f}: \mathcal{F} \rightarrow \R$ is an orthonormally invariant function, then
\begin{align}
\begin{cases}
H_{a}H_{b}\tilde{f} - H_{b}H_{a}\tilde{f} &= 0 \\
\Delta_{H}H_{a}\tilde{f} - H_{a}\Delta_{H}\tilde{f} &= \widetilde{\mathrm{Rc}}_{ab}H_{b}\tilde{f},
\end{cases}
\end{align}
where $\widetilde{\mathrm{Rc}}_{ab}(u) = \mathrm{Rc}_{\pi(u)}(ue_{a},ue_{b})$.
\end{corollary}
\begin{proof}
This is a direct application of the commutator relations from Lemma \ref{1.6}.
\end{proof}
\begin{proposition} \label{1.8}
Let $\tilde{f}: \mathcal{F} \rightarrow \R$ be an orthonormally invariant function. Then
\begin{align}
[D_{t}-\Delta_{H},H_{a}]\tilde{f} = -\frac{1}{2}(\widetilde{\partial_{t}g}+2\widetilde{\mathrm{Rc}})_{ab}H_{b}\tilde{f}.
\end{align}
\end{proposition}
\begin{proof}
It readily follows from Lemma \ref{1.6} and Corollary \ref{1.7} that
\begin{align}
[D_{t}-\Delta_{H},H_{a}]\tilde{f} &= [D_{t},H_{a}]\tilde{f} - [\Delta_{H},H_{a}]\tilde{f} \nonumber \\
&= -\frac{1}{2}\widetilde{\partial_{t}g}_{ad}H_{d}\tilde{f} + \frac{1}{2}H_{b}(\widetilde{\partial_{t}g}_{ac})V_{cb}\tilde{f} - \widetilde{\mathrm{Rc}}_{ab}H_{b}\tilde{f} \nonumber \\
&= -\frac{1}{2}(\widetilde{\partial_{t}g}+2\widetilde{\mathrm{Rc}})_{ab}H_{b}\tilde{f},
\end{align}
thereby proving the proposition.
\end{proof}

\pagebreak

\subsection{Probabilistic Preliminaries}

The principal goal of this section is to recall the notions of Brownian motion and stochastic parallel transport in the setting of evolving manifolds as developed in \cite{ACT08} and \cite{HN18a}.
\newline \\
We first remark that it shall hereafter be assumed that in addition to the Riemannian manifolds $\left \{M_{t} \right \}$ being complete as in the previous section, they will also satisfy
\begin{align}
\sup_{\mathcal{M}} \left (|\mathrm{Rm}| + |\partial_{t}g|  + |\nabla \partial_{t}g| \right ) < \infty.
\end{align}
Horizontal curves $\left \{u_{\tau} \right \}_{\tau \in [0,T]} \in \mathcal{F}$, where $\pi(u_{\tau}) = (x_{\tau},T-\tau)$, correspond to curves $\left \{w_{\tau} \right \}_{\tau \in [0,T]} \in \R^{n}$ (also known as the anti-development of $u_{\tau}$) via the following initial value problem
\begin{align}
\begin{cases} \frac{du_{\tau}}{d\tau} &= D_{\tau} + H_{a}(u_{\tau}) \frac{dw_{\tau}^{a}}{d\tau} \\
w_{0} &= 0. 
\end{cases}
\end{align}
\noindent This definition of the anti-development in the time-dependent geometry setting appropriately motivates the following stochastic differential equation in the case of evolving manifolds
\begin{align}
\begin{cases} dU_{\tau} &= D_{\tau}\, d\tau + H_{a}(U_{\tau}) \circ dW_{\tau}^{a} \label{49} \\
U_{0} &= u.
\end{cases}
\end{align}
We make a short note on notation that $W_{\tau} \sim \sqrt{2}B_{\tau}$ refers to the Brownian motion in $\R^{n}$ with rescaling by a factor of $\sqrt{2}$ such that it has quadratic variation $d[W,W]_{\tau} = 2d[B, B]_{\tau} = 2\, d\tau$ and $\circ$ refers to the Stratonovich integral in differential notation.
\newline \\
Next, by demonstrating that this equation satisfies existence and uniqueness criterion as well as It\^o's lemma, the notions of Brownian motion, via projection onto $\mathcal{M}$, and stochastic parallel transport can be formalized.
\begin{proposition}
(Existence, uniqueness and It\^o's lemma) \label{50} \label{1.9}
The stochastic differential equation \eqref{49} has a unique solution $\left\{U_{\tau} \right \}_{\tau \in [0,T]}$ that satisfies $\pi_{2}(U_{\tau}) = T-\tau$. Moreover, given $\tilde{f}: \mathcal{F} \rightarrow \R$ is of class $C^2$, then the solution $U_{\tau}$ satisfies
\begin{align}
d\tilde{f}(U_{\tau}) = D_{\tau} \tilde{f}(U_{\tau})\, d\tau + \langle (H\tilde{f})(U_{\tau}), dW_{\tau} \rangle + \Delta_{H}(\tilde{f})(U_{\tau})\, d\tau.
\end{align}
\end{proposition} 

\begin{proof}
This result has already been proven in Proposition 3.7 of \cite{HN18b}.
\end{proof}
\noindent We shall now continue with defining the notions of Brownian motion and stochastic parallel transport, from \cite{HN18a}, in the setting of time-evolving families of Riemannian manifolds.
\begin{definition} (Brownian motion on space time)
We call $\pi(U_{\tau}) = (X_{\tau},T-\tau)$ the Brownian motion on space time $\mathcal{M} = M \times I$ with base point $\pi(u)=(x,T)$.
\end{definition}
\begin{definition} (Stochastic parallel transport)
The family of isometries
\begin{align}
\left \{P_{\tau} = U_{0}U_{\tau}^{-1} : (T_{X_{\tau}}M,g_{T-\tau}) \rightarrow (T_{x}M,g_{T}) \right \}
\end{align}
is called the stochastic parallel transport along the Brownian curve $X_{\tau}$.
\end{definition}
\noindent This Brownian motion now inherits a path based space, diffusion measure and filtration. First, we denote by $P_{0}\R^{n}$ the based path space on $\R^{n}$, namely the space of continuous curves $\left \{w_{\tau} | w_{0}=0 \right \}_{\tau \in [0,T]} \subset \R^{n}$.
\begin{definition} (Based path spaces)
Let $P_{u}\mathcal{F}$ and $P_{(x,T)}\mathcal{M}$ be the spaces of continuous curves, $\left \{u_{\tau} | u_{0}=u, \pi_{2}(u_{\tau})= T-\tau \right \}_{\tau \in [0,T]} \subset \mathcal{F}$ and $\left \{\gamma_{\tau} = (x_{\tau},T-\tau) | \gamma_{0}=(x,T) \right \}_{\tau \in [0,T]}$ respectively. 
\end{definition}
\noindent To construct the Wiener measure, we first observe that solving the stochastic differential equation \eqref{49} yields a map $U: P_{0}\R^{n} \rightarrow P_{u}\mathcal{F}$. Moreover, the projection map $\pi: \mathcal{F} \rightarrow \mathcal{M}$ induces a map $\Pi: P_{u}\mathcal{F} \rightarrow P_{(x,T)}\mathcal{M}$.
\begin{definition} (Wiener measure)
Let $\mathbb{P}_{0}$ be the Wiener measure on path space $P_{0}\R^{n}$. We then say that $\mathbb{P}_{u} := U_{*}(\mathbb{P}_{0})$ and $\mathbb{P}_{(x,T)} := \Pi_{*}\mathbb{P}_{u}$ are the Wiener measures of horizontal Brownian motion on $\mathcal{F}$ and Brownian motion on space-time $\mathcal{M}$ respectively.
\end{definition}
\noindent Moreover, we can uniquely characterize the Wiener measure in terms of the heat kernel.
\begin{proposition} \cite{HN18a} \label{2.14}
Let $\left \{\tau_{j} \right \}_{j=1}^{k}$ be a partition of $[0,T]$, $U_{j} \subseteq M$ and $\gamma_{0}=(x,T)$. Then 
\begin{align}
\mathbb{P}_{(x,T)} \left [X_{\tau_{j}} \in U_{j}, \, \forall j \in {1,...,k} \right] = \int_{\times_{1}^{k} U_{j}} \Pi_{j} H(x_{j-1}, T-\tau_{j-1} | x_{j}, T-\tau_{j}) \otimes d\mathrm{Vol}_{g_{T-\tau_{j}}}(y_{j})
\end{align}
uniquely characterizes the Wiener measure on $P_{(x,T)}\mathcal{M}$.
\end{proposition}
\begin{proof}
The proof follows as in Proposition 3.31 of \cite{HN18a}.
\end{proof}
\noindent Next, we recall that the path space $P_{0}\R^{n}$ comes equipped with an intrinsic filtration $\Sigma_{\tau}^{\R^{n}}$ generated by evaluation maps $\left \{e_{\sigma}: P_{0}\R^{n} \rightarrow \R^{n} | e_{\sigma}(w) = w_{\sigma}\, , \sigma \leq \tau \right \}$.
\begin{definition} (Filtrations on $P_{u}\mathcal{F}$ and $P_{(x,T)}\mathcal{M}$)
The filtrations on $P_{u}\mathcal{F}$ and $P_{(x,T)}\mathcal{M}$ are simply the respective push-forwards $\Sigma_{\tau}^{\mathcal{M}} := (\Pi \circ U)_{*} \Sigma_{\tau}^{\R^{n}}$ and $\Sigma_{\tau}^{\mathcal{F}} := U_{*} \Sigma_{\tau}^{\R^{n}}$.
\end{definition}
\noindent A short reiteration of induced martingales as well as parallel and Malliavin gradients are constructed in the time-dependent setting will now complete this section.
\begin{definition} (Induced martingale)
Let $F: P_{(x,T)}\mathcal{M} \rightarrow \R$ be integrable. Then, we define the induced martingale as $F_{\tau}(\gamma) := \mathbb{E}[F | \Sigma_{\tau}](\gamma)$.
\end{definition}
\noindent Using this definition, the conditional expectation can now be characterized by a representation formula.
\begin{proposition} (Conditional expectation) \cite{HN18a}  \label{1.17}
Suppose the conditional expectation is as defined. Then, for almost every Brownian curve $\left \{\gamma_{\tau} \right \}_{\tau \in [0,T]}$,
\begin{align}
F_{\tau}(\gamma) := \mathbb{E}[F | \Sigma_{\tau}](\gamma) = \int_{P_{\gamma_{\tau}}\mathcal{M}} F(\gamma |_{[0, \tau]} * \gamma') d\mathbb{P}_{\gamma_{\tau}}(\gamma'),
\end{align}
where we integrate over all Brownian curves $\gamma'$ in the based path space $P_{\gamma_{\tau}}\mathcal{M}$ with respect to Wiener measure $\mathbb{P}_{\gamma_{\tau}}$ and $*$ denotes concatenation of the two curves $\gamma |_{[0,\tau]}$ and $\gamma'$.
\end{proposition}
\begin{proof}
The proof follows as in Proposition 3.19 of \cite{HN18a}.
\end{proof}
\noindent To define the two notions of gradients, we first recall that cylinder functions are of the form $u \circ e_{\bm{\sigma}}$, where $e_{\bm{\sigma}} : P_{(x,T)}\mathcal{M} \rightarrow M^{k}$ are $k$-point evaluation maps, namely $e_{\bm{\sigma}} : \gamma \mapsto (\pi_{1} \gamma_{\sigma_{1}},...,\pi_{1} \gamma_{\sigma_{k}})$, and $u: M^{k} \rightarrow \R$ is compactly supported.
\begin{example} \label{example 2.18}
Let $F(\gamma) := f \circ e_{\tau_{1}}(\gamma) = f(\pi_{1} \gamma_{\tau_{1}})$. Then the induced martingale of $F$ is given for $\tau > \tau_{1}$ by 
\begin{align}
F_{\tau}(\gamma) &= \int_{P_{\gamma_{\tau}}\mathcal{M}} F(\gamma|_{[0,\tau]}*\gamma')\, d\mathbb{P}_{\gamma_{\tau}}(\gamma') \nonumber \\
&= \int_{P_{\gamma_{\tau}}\mathcal{M}} f(\pi_{1}\gamma_{\tau_{1}})\, d\mathbb{P}_{\gamma_{\tau}}(\gamma')  \nonumber \\
&= f(X_{\tau_{1}}), 
\end{align}
and for $\tau < \tau_{1}$ by
\begin{align}
F_{\tau}(\gamma) &= \int_{P_{\gamma_{\tau}}\mathcal{M}} F(\gamma|_{[0,\tau]}*\gamma')\, d\mathbb{P}_{\gamma_{\tau}}(\gamma') \nonumber \\ 
&= \int_{P_{\gamma_{\tau}}\mathcal{M}} f(\pi_{1}\gamma'_{\tau_{1}-\tau})\, d\mathbb{P}_{\gamma_{\tau}}(\gamma') \nonumber \\
&= \int_{M} f(y)H(X_{\tau},T-\tau| y,T-\tau_{1})\, dV_{g_{T-\tau_{1}}}(y) \nonumber \\
&= H_{T-\tau_{1},T-\tau}f(\pi_{1} \gamma_{\tau}).
\end{align}
\end{example}
\begin{definition} (Parallel gradient) \label{2.19}
Let $\sigma \in [0,T]$ and let $F: P_{(x,T)}\mathcal{M} \rightarrow \R$ be a cylinder function. Then the $\sigma$-parallel gradient is the almost everywhere uniquely defined function $\nabla_{\sigma}^{||}F: P_{(x,T)}\mathcal{M} \rightarrow (T_{x}M, g_{T})$ such that
\begin{align}
D_{V^{\sigma}}F(\gamma) = \langle \nabla_{\sigma}^{||}F(\gamma), v \rangle_{(T_{x}M,g_{T})},
\end{align}
for almost every Brownian curve $\gamma$ and $v \in (T_{x}M, g_{T})$, where $V_{\tau}^{\sigma} = P_{\tau}^{-1}v \mathbbm{1}_{[\sigma, T]}(\tau)$. Here, $D_{V}$ denotes the Fr\'echet derivative.
\end{definition}
\begin{example} \label{1.19}
Let $F = u \circ e_{\bm{\tau}}$ be a $k$-point function with partition $\left \{\tau_{j} \right \}_{j=1}^{k}$. Then the parallel gradient of $F$ is given by
\begin{align}
\nabla_{\sigma}^{||}F &= e_{\bm{\tau}}^{*} \left (\sum_{\tau_{j} \geq \sigma} P_{\tau_{j}}\mathrm{grad}_{g_{T-\tau_{j}}}^{(j)}u \right ).
\end{align}
\end{example}
\noindent Finally, we let $\mathcal{H}$ be the Hilbert space of $W^{1,2}$ curves in $(T_{x}M,g_{T})$ with $v_{0}=0$ equipped with the natural Sobolev inner product given by
\begin{align}
\langle u, v \rangle_{\mathcal{H}} := \int_{0}^{T} \langle \dot{u}_{\tau}, \dot{v}_{\tau} \rangle _{(T_{x}M,g_{T})}\, d\tau.
\end{align}
\begin{definition} (Malliavin gradient)
Let $F: P_{(x,T)}\mathcal{M} \rightarrow \R$ be a cylinder function. The Malliavin gradient is the almost everywhere uniquely defined function $\nabla^{\mathcal{H}}F: P_{(x,T)}\mathcal{M} \rightarrow \mathcal{H}$ such that
\begin{align}
D_{V}F(\gamma) = \langle \nabla^{\mathcal{H}}F(\gamma), v \rangle_{\mathcal{H}}, \label{Malliavin}
\end{align}
for almost every Brownian curve $\gamma$ and $v \in \mathcal{H}$, where $V_{\tau} = P_{\tau}^{-1}v_{\tau}$. 
\end{definition}
\begin{definition} (Skorokhod integral) The adjoint of the Malliavin gradient, also known as the Skorokhod integral, is the uniquely defined operator $(\nabla^{\mathcal{H}})^{*}: L^{2}(P\mathcal{M}) \to L^{2}(P\mathcal{M})$ such that
\begin{align}
\mathbb{E}[F (\nabla^{\mathcal{H}})^{*}g] = \mathbb{E} \left [\langle \nabla^{\mathcal{H}}F, g \rangle_{\mathcal{H}} \right ] , \label{adjoint}
\end{align}
for all $F,g \in L^{2}(P\mathcal{M})$.
\end{definition}
\begin{definition} (Ornstein-Uhlenbeck operator) The Ornstein-Uhlenbeck operator is defined as 
\begin{align}
\mathcal{L}_{(\tau_{1},\tau_{2})} := (\nabla^{\mathcal{H}})^{*}\nabla^{\mathcal{H}}
\end{align}
where $\nabla^{\mathcal{H}}$ and $(\nabla^{\mathcal{H}})^{*}$ are the Malliavin gradient and its adjoint from \eqref{Malliavin} and \eqref{adjoint} respectively.
\end{definition}

\pagebreak

\section{Bochner Formula on Parabolic Path Space}

For convenience of the reader, we shall first recall and prove the statement of Bochner's formula in the time-dependent setting.

\begin{lemma} (Bochner) First let $(M,g_{t})_{t\in I}$ be a family of Riemannian manifolds and $\mathrm{grad}_{g_{t}}^{i} = g_{t}^{ij}\partial_{j}$ where $\Delta_{g_{t}}$ be the Laplace-Beltrami operator. Then the evolution of $|\nabla u|_{g_{t}}^{2}$ is given by \label{2.1}
\begin{align}
&\frac{1}{2}(-\partial_{t}+\Delta_{g_{t}})(|\nabla u|^2_{g_{t}}) \\
&\quad = \langle \nabla u, \nabla(-\partial_{t}+\Delta_{g_{t}})u \rangle + |\nabla^2 u|^2 + \frac{1}{2}(\partial_{t}g_{t}+2\mathrm{Rc}_{g_{t}})(\mathrm{grad}\, u, \mathrm{grad}\, u). \nonumber
\end{align}
\end{lemma}

\begin{proof}
We evaluate both
\begin{align}
\frac{1}{2}\Delta_{g_{t}}(|\nabla u|^2_{g_{t}}) &= \frac{1}{2} \nabla_{i} \nabla_{i} (\nabla_{j} u \nabla_{j} u) \nonumber \\
&= \left (\nabla_{i}\nabla_{j} u \right )\left (\nabla_{i}\nabla_{j} u \right ) + (\nabla_{j}u) \left(\nabla_{i} \nabla_{i} \nabla_{j} u \right ) \nonumber \\
&= |\nabla^2 u|^2 + (\nabla_{j}u) (\nabla_{j}\Delta_{g_{t}} u) + \mathrm{Rc}_{g_{t}}(\mathrm{grad}\, u, \mathrm{grad}\, u)
\end{align}
and
\begin{align}
\frac{1}{2}\partial_{t}(|\nabla u|^2_{g_{t}}) &= \frac{1}{2}\partial_{t}g_{t}^{ij}\nabla_{i}u\nabla_{j}u + g_{t}^{ij} \nabla_{i}u \partial_{t}(\nabla_{j}u) \nonumber \\
&= -\frac{1}{2}\partial_{t}g_{k\ell}g^{ki}g^{\ell j} \nabla_{i}u \nabla_{j}u + g_{t}^{ij} \nabla_{i} u \partial_{t}(\nabla_{j} u) \nonumber \\
&= \langle \nabla u, \nabla (\partial_{t} u) \rangle - \frac{1}{2}\partial_{t}g_{t}(\mathrm{grad}\, u, \mathrm{grad}\, u).
\end{align}
We then deduce the Bochner formula as the difference of the two results. 
\end{proof}

\begin{theorem} (Martingale representation theorem) \label{3.2}
If $F_{\tau} : P_{(x,T)} \mathcal{M} \rightarrow \R$ is a martingale on parabolic path space and $F_{\tau} \in \mathcal{D}(\nabla_{\tau}^{||})$, then $F_{\tau}$ solves stochastic differential equation
\begin{align}
\begin{cases}
dF_{\tau} = \langle \nabla_{\tau}^{||}F_{\tau}, dW_{\tau} \rangle \\
F|_{\tau=0} = F_{0}.
\end{cases}
\end{align}
\end{theorem}

\begin{proof}
By approximation (cf. \cite[Sec 2.4]{HN18a}), it suffices to prove the theorem in the case where $F_{\tau}$ is a martingale induced by a $k$-point cylinder function. Namely, let $F(\gamma) = f(\pi_{1}\gamma_{\tau_{1}},..., \pi_{1}\gamma_{\tau_{k}})$, where $f: M^{k} \rightarrow \R$ and we recall that $\gamma_{\tau}=(X_{\tau},T-\tau)$. Also let $F_{\tau} = \mathbb{E}_{(x,T)}[F | \Sigma_{\tau}]$ be the induced martingale. Then, for $\tau \in (\tau_{\l},\tau_{\l+1})$ by Propositions \ref{1.17} and then \ref{2.14}, we calculate
\begin{align}
F_{\tau}(\gamma) &= \int_{P_{\gamma_{\tau}}\mathcal{M}} F(\gamma|_{[0,\tau]}*\gamma')\, d\mathbb{P}_{\gamma_{\tau}}(\gamma') \nonumber \\
&= \int_{P_{\gamma_{\tau}}\mathcal{M}} f(\pi_{1}\gamma_{\tau_{1}},...,\pi_{1}\gamma_{\tau_{\l}}, \pi_{1}\gamma_{\tau_{\l+1}-\tau}^{'},...,\pi_{1}\gamma_{\tau_{k}-\tau}^{'})\, d\mathbb{P}_{\gamma_{\tau}}(\gamma') \nonumber \\
&= \int_{M^{k-\l}} f(X_{\tau_{1}},\ldots,X_{\tau_{\l}},y_{\l+1},\ldots, y_{k})H(X_{\tau},T-\tau|y_{\l+1},T-\tau_{\l+1}) \nonumber \\
&\quad H(y_{\l+1},T-\tau_{\l+1}|y_{\l+2},T-\tau_{\l+2})\cdots H(y_{k-1},T-\tau_{k-1}|y_{k},T-\tau_{k}) \nonumber \\
&\quad \, dV_{g_{T-\tau_{\l+1}}}(y_{\l+1}) \cdots dV_{g_{T-\tau_{k}}}(y_{k}) \nonumber \\
&=: f_{\tau}(X_{\tau_{1}},..,X_{\tau_{\l}},X_{\tau}).
\end{align}
Note that, for $(x_{1},\ldots, x_{\l})$ fixed, $(x,\tau) \rightarrow f_{\tau}(x_{1},\ldots,x_{\l},x)$ is uniformly Lipschitz in $\tau$ and solves $(\partial_{\tau}+\Delta^{(\l+1)})f_{\tau} = 0$, where $\Delta^{(\l+1)}$ acts on the last entry.
\newline \\
Consider the lift $\tilde{f}_{\tau}: = f_{\tau} \circ \otimes_{1}^{\l+1} \pi_{1} \circ \otimes_{1}^{\l+1}\pi$. Also let $\tilde{F_{\tau}} := F_{\tau} \circ \Pi$, where $\Pi: P\mathcal{F} \rightarrow P\mathcal{M}$. Then we have that $\tilde{F}_{\tau}(U) = \tilde{f}_{\tau}(U_{\tau_{1}}, ..., U_{\tau_{\l}},U_{\tau})$, which satisfies $(D_{\tau} + \Delta_{H}^{(\l+1)})\tilde{f}_{\tau} = 0$ by applying Proposition \ref{derivatives}. Also note that herein we shall denote the vector $(H_{1}\tilde{f},...H_{n}\tilde{f})$ by $H\tilde{f}$.
\newline \\
Then, by Proposition \ref{1.9}, we calculate
\begin{align}
d\tilde{F_{\tau}}(U) = d(\tilde{f}_{\tau}(U_{\tau_{1}},...,U_{\tau_{\l}},U_{\tau})) &= \langle H^{(\l+1)}(\tilde{f})(U_{\tau_{1}},...,U_{\tau_{\l}},U_{\tau}), dW_{\tau} \rangle \nonumber \\
&\quad + \left (D_{\tau}+ \Delta_{H}^{(\l+1)} \right )\tilde{f}_{\tau}(U_{\tau_{1}},...,U_{\tau_{\l}},U_{\tau}) \, d\tau \nonumber \\
&= \langle H^{(\l+1)}(\tilde{f})(U_{\tau_{1}},...,U_{\tau_{\l}},U_{\tau}), dW_{\tau} \rangle.
\end{align}
Next, we project down to $M$ by Proposition \ref{derivatives} as follows
\begin{align}
H_{a}^{(\l+1)} \tilde{f_{\tau}}(U_{\tau_{1}},...,U_{\tau_{\l}},U_{\tau}) &= (U_{\tau}e_{a})^{*}\tilde{f_{\tau}}(U_{\tau_{1}},...,U_{\tau_{\l}},U_{\tau}) \nonumber \\
&= (U_{\tau}e_{a})f_{\tau}(X_{\tau_{1}},...X_{\tau_{\l}},X_{\tau}) \nonumber \\
&= \langle U_{\tau}e_{a}, \mathrm{grad}_{g_{T-\tau}}^{(\l+1)}f_{\tau}(X_{\tau_{1}},...,X_{\tau_{\l}},X_{\tau}) \rangle _{(T_{X_{\tau}}M,g_{T-\tau})} \nonumber \\
&= \langle P_{\tau}U_{\tau}e_{a}, P_{\tau}\mathrm{grad}_{g_{T-\tau}}^{(\l+1)}f_{\tau}(X_{\tau_{1}},...,X_{\tau_{\l}},X_{\tau}) \rangle _{(T_{x}M,g_{T})} \nonumber \\
&= \langle U_{0}e_{a}, \nabla_{\tau}^{||}F_{\tau}(\gamma) \rangle _{(T_{x}M,g_{T})},
\end{align}
whence
\begin{align}
H_{a}^{(\l+1)}(\tilde{f_{\tau}})\, dW_{\tau}^{a} = \langle \nabla_{\tau}^{||}F_{\tau}(\gamma), U_{0}e_{a} \rangle\, dW_{\tau}^{a} = \langle \nabla_{\tau}^{||}F_{\tau}(\gamma), dW_{\tau} \rangle,    
\end{align}
and we deduce that 
\begin{align}
dF_{\tau}(\gamma) = d\tilde{F_{\tau}}(U) = \langle \nabla_{\tau}^{||}F_{\tau}(\gamma), dW_{\tau} \rangle
\end{align}
to complete the proof.
\end{proof}

\begin{theorem} (Evolution of the parallel gradient) \label{3.3}
If $F_{\tau} : P_{(x,T)}\mathcal{M} \rightarrow \R$ is a martingale on parabolic path space, and $\sigma \geq 0$ is fixed, then $\nabla_{\sigma}^{||}F_{\tau} : P_{(x,T)}\mathcal{M} \rightarrow (T_{(x,T)}M,g_{T})$ satisfies the stochastic differential equation
\begin{align}
d(\nabla_{\sigma}^{||}F_{\tau}) = \langle \nabla_{\tau}^{||}\nabla_{\sigma}^{||}F_{\tau}, dW_{\tau} \rangle + \frac{1}{2}(\dot{g}+2\mathrm{Rc})_{\tau}(\nabla_{\tau}^{||}F_{\tau})\, d\tau + \nabla_{\sigma}^{||}F_{\sigma}d\delta_{\sigma}(\tau),
\end{align}
where $\langle (\dot{g}+2\mathrm{Rc})_{\tau}(v), w \rangle_{(T_{x}M,g_{T})} = (\dot{g_{t}}+2\mathrm{Rc}_{g_{t}})|_{t=T-\tau}(P_{\tau}^{-1}v,P_{\tau}^{-1}w)$ and $\dot{g} = \frac{d}{dt}g$.
\end{theorem}

\begin{proof}
As $F_{\tau}$ is $\Sigma_{\tau}$-measurable, we have that $\nabla_{\sigma}^{||}F_{\tau} \equiv 0$ for $\sigma>\tau$. Noting $d(\nabla_{\sigma}^{||}F_{\tau})$ is continuous except for a jump discontinuity at $\sigma=\tau$,  we calculate
\begin{align}
d(\nabla_{\sigma}^{||}F_{\tau}) &= d(\nabla_{\sigma}^{||}F_{\tau})_{\mathrm{cont}} + \left (\nabla_{\sigma}^{||}F_{\sigma^{+}} - \nabla_{\sigma}^{||}F_{\sigma^{-}}  \right )d\delta_{\sigma}(\tau) \nonumber \\
&= d(\nabla_{\sigma}^{||}F_{\tau})_{\mathrm{cont}} + \nabla_{\sigma}^{||}F_{\sigma} \, d\delta_{\sigma}(\tau).
\end{align}
It remains to show that the identity holds for $\sigma \leq \tau$. In particular, we'll show that the continuous parts of the measures agree.
\newline \\
\noindent By approximation (cf. \cite[Sec 2.4]{HN18a}), it suffices to prove the theorem in the case where $F_{\tau}$ is a martingale induced by a $k$-point cylinder function as in the previous proof. Now, as $\sigma$ is fixed, it is sufficient for us to consider the evolution equation for $\tau \in (\tau_{\l}, \tau_{\l+1})$, using the parallel gradient from example \ref{1.19},
\begin{align}
\nabla_{\sigma}^{||}F_{\tau}(\gamma) = \sum_{\tau_{j} \geq \sigma} P_{\tau_{j}}\nabla^{(j)}f_{\tau}(X_{\tau_{1}},..., X_{\tau_{\l}},X_{\tau}) + P_{\tau}\nabla^{(\l+1)}f_{\tau}(X_{\tau_{1}},..., X_{\tau_{\l}},X_{\tau}),
\end{align}
which can be lifted to the frame bundle and represented by Proposition \ref{derivatives} as
\begin{align}
G_{a}(U) :&= \langle U_{0}e_{a}, \nabla_{\sigma}^{||}F_{\tau}(\Pi U) \rangle \nonumber \\
&= \sum_{\tau_{j} \geq \sigma} \langle U_{\tau_{j}}e_{a}, \nabla^{(j)} f_{\tau}(X_{\tau_{1}},...,X_{\tau_{\l}},X_{\tau}) \rangle \nonumber \\
&\quad + \langle U_{\tau}e_{a}, \nabla^{(\l+1)}f_{\tau}(X_{\tau_{1}},...,X_{\tau_{\l}},X_{\tau}) \rangle \nonumber \\
&= \sum_{\tau_{j} \geq \sigma} H_{a}^{(j)}\tilde{f_{\tau}}(U_{\tau_{1}},..., U_{\tau_{\l}},U_{\tau}) + H_{a}^{(\l+1)}\tilde{f_{\tau}}(U_{\tau_{1}},..., U_{{\tau}_{\l}},U_{\tau}).
\end{align}
Applying Propositions \ref{1.8} and \ref{1.9} and the fact that $(D_{\tau}+\Delta_{H}^{(\l+1)})\tilde{f}_{\tau} = 0$ from the previous proof, we have that
\begin{align}
dG_{a}(U) &= \sum_{\tau_{k}\geq \sigma} H_{b}^{(\l+1)}H_{a}^{(k)} \tilde{f_{\tau}}(U_{\tau_{1}},...,U_{\tau_{\l}},U_{\tau})\, dW_{\tau}^{b} \nonumber \\
&\quad + H_{b}^{(\l+1)}H_{a}^{(\l+1)}\tilde{f_{\tau}}(U_{\tau_{1}},...,U_{\tau_{\l}},U_{\tau})\, dW_{\tau}^{b} \nonumber \\
&\quad + \sum_{\tau_{k}\geq \sigma} \left (D_{\tau}+\Delta_{H}^{(\l+1)} \right )H_{a}^{(k)}\tilde{f_{\tau}}(U_{\tau_{1}},...,U_{\tau_{\l}},U_{\tau})\, d\tau \nonumber \\
&\quad + \left (D_{\tau}+\Delta_{H}^{(\l+1)} \right )H_{a}^{(\ell+1)}\tilde{f_{\tau}}(U_{\tau_{1}},...,U_{\tau_{\l}},U_{\tau})\, d\tau \nonumber \\
&= \sum_{\tau_{k}\geq \sigma} H_{b}^{(\l+1)}H_{a}^{(k)} \tilde{f_{\tau}}(U_{\tau_{1}},...,U_{\tau_{\l}},U_{\tau})\, dW_{\tau}^{b} \nonumber \\
&\quad + H_{b}^{(\l+1)}H_{a}^{(\l+1)}\tilde{f_{\tau}}(U_{\tau_{1}},...,U_{\tau_{\l}},U_{\tau})\, dW_{\tau}^{b} \nonumber \\
&\quad + \sum_{\tau_{k}\geq \sigma} \left (H_{a}^{(k)} \left (D_{\tau}+\Delta_{H}^{(\l+1)} \right )+[D_{\tau}+\Delta_{H}^{(\l+1)},H_{a}^{(k)}] \right )\tilde{f_{\tau}}(U_{\tau_{1}},...,U_{\tau_{\l}},U_{\tau})\, d\tau \nonumber \\
&\quad + \left (H_{a}^{(\l+1)}\left (D_{\tau}+\Delta_{H}^{(\l+1)} \right )+[D_{\tau}+\Delta_{H}^{(\l+1)},H_{a}^{(\l+1)}] \right )\tilde{f_{\tau}}(U_{\tau_{1}},...,U_{\tau_{\l}},U_{\tau})\, d\tau \nonumber \\
&= \sum_{\tau_{k}\geq \sigma} H_{b}^{(\l+1)}H_{a}^{(k)} \tilde{f_{\tau}}(U_{\tau_{1}},...,U_{\tau_{\l}},U_{\tau})\, dW_{\tau}^{b} \\
&\quad + H_{b}^{(\l+1)}H_{a}^{(\l+1)}\tilde{f_{\tau}}(U_{\tau_{1}},...,U_{\tau_{\l}},U_{\tau})\, dW_{\tau}^{b} \nonumber \\
&\quad + \frac{1}{2}(\widetilde{\dot{g}}+2\widetilde{\mathrm{Rc}})_{ab}(U_{\tau}) H_{b}^{(\l+1)} \tilde{f_{\tau}}(U_{\tau_{1}},...,U_{\tau_{\l}},U_{\tau})\, d\tau \nonumber.
\end{align}
Finally, we project down onto $\mathcal{M}$ by Proposition \ref{derivatives} as follows,
\begin{align}
&H_{b}^{(\l+1)}H_{a}^{(\l+1)}\tilde{f}_{\tau}(U_{\tau_{1}},...,U_{\tau_{\l}},U_{\tau}) \nonumber \\
&\quad = (U_{\tau}e_{b})^{*}(U_{\tau}e_{a})^{*}\tilde{f}_{\tau}(U_{\tau_{1}},...,U_{\tau_{\l}},U_{\tau}) \nonumber \\
&\quad = \left (\nabla^{(\l+1)}\nabla^{(\l+1)}f_{\tau} (X_{\tau_{1}},...,X_{\tau_{\l}},X_{\tau}) \right ) (U_{\tau}e_{b},U_{\tau}e_{a}) \nonumber \\
&\quad = \langle U_{\tau}e_{b}\otimes U_{\tau}e_{a}, \nabla^{(\l+1)}\nabla^{(\l+1)}f_{\tau}(X_{\tau_{1}},...,X_{\tau_{\l}},X_{\tau}) \rangle \nonumber \\
&\quad = \big \langle U_{0}e_{b} \otimes U_{0}e_{a}, (P_{\tau} \otimes P_{\tau}) \left (\nabla^{(\l+1)}\nabla^{(\l+1)}f_{\tau}(X_{\tau_{1}},...,X_{\tau_{\l}},X_{\tau}) \right ) \! \big \rangle,
\end{align}
and similarly,
\begin{align}
&H_{b}^{(\l+1)}H_{a}^{(k)}\tilde{f}_{\tau}(U_{\tau_{1}},...,U_{\tau_{\l}},U_{\tau}) \nonumber \\
&\quad = \langle U_{\tau}e_{b}\otimes U_{\tau_{k}}e_{a}, \nabla^{(\l+1)}\nabla^{(k)}f_{\tau}(X_{\tau_{1}},...,X_{\tau_{\l}},X_{\tau}) \rangle \nonumber \\
&\quad = \big \langle U_{0}e_{b} \otimes U_{0}e_{a}, (P_{\tau} \otimes P_{\tau_{k}}) \left (\nabla^{(\l+1)}\nabla^{(k)}f_{\tau}(X_{\tau_{1}},...,X_{\tau_{\l}},X_{\tau}) \right ) \! \big \rangle,
\end{align}
whence
\begin{align}
&\sum_{\tau_{k}\geq \sigma}(H_{b}^{(\l+1)}H_{a}^{(k)})(\tilde{f}_{\tau})\, dW_{\tau}^{b} + (H_{b}^{(\l+1)}H_{a}^{(\l+1)})(\tilde{f}_{\tau})\, dW_{\tau}^{b} \nonumber \\
&\quad = \bigg \langle \sum_{\tau_{k}\geq \sigma} (P_{\tau} \otimes P_{\tau_{k}})\nabla^{(\l+1)}\nabla^{(k)}f_{\tau}, U_{0}e_{b}\otimes U_{0}e_{a} \bigg \rangle\, dW_{\tau}^{b} \nonumber \\
&\quad \quad + \langle (P_{\tau} \otimes P_{\tau})\nabla^{(\l+1)}\nabla^{(\l+1)}f_{\tau}, U_{0}e_{b}\otimes U_{0}e_{a}\rangle\, dW_{\tau}^{b} \nonumber \\
&\quad = \bigg \langle \sum_{\tau_{k}\geq \sigma} (P_{\tau} \otimes P_{\tau_{k}})\nabla^{(\l+1)}\nabla^{(k)}f_{\tau}, dW_{\tau}\otimes U_{0}e_{a} \bigg \rangle \nonumber \\
&\quad \quad + \langle (P_{\tau} \otimes P_{\tau})\nabla^{(\l+1)}\nabla^{(\l+1)}f_{\tau}, dW_{\tau}\otimes U_{0}e_{a} \rangle \nonumber \\
&\quad = \bigg \langle \sum_{\tau_{k}\geq \sigma} (P_{\tau} \otimes P_{\tau_{k}})\nabla^{(\l+1)}\nabla^{(k)}f_{\tau} + (P_{\tau} \otimes P_{\tau})\nabla^{(\l+1)}\nabla^{(\l+1)}f_{\tau}, dW_{\tau}\otimes U_{0}e_{a} \bigg \rangle \nonumber \\
&\quad= \langle \nabla_{\tau}^{||}\nabla_{\sigma}^{||}F_{\tau}(\gamma), dW_{\tau}\otimes U_{0}e_{a} \rangle.
\end{align}
Finally, we check that
\begin{align}
&(\widetilde{\dot{g}}+2\widetilde{\mathrm{Rc}})_{ab}(U_{\tau})H_{b}^{(\l+1)}\tilde{f}_{\tau}(U_{\tau_{1}},...,U_{\tau_{\l}},U_{\tau})\, d\tau \nonumber \\
&\quad = (\dot{g}+2\mathrm{Rc})_{\pi(U_{\tau})}(U_{\tau}e_{a}, U_{\tau}e_{b}) \langle \nabla^{(\ell+1)}f_{\tau}(X_{\tau_{1}},\dots,X_{\tau_{\ell}},X_{\tau}), U_{\tau}e_{b} \rangle\, d\tau \nonumber \\ 
&\quad = (\dot{g}+2\mathrm{Rc})_{\pi(U_{\tau})} \left (U_{\tau}e_{a}, \langle \nabla^{(\ell+1)}f_{\tau}(X_{\tau_{1}},\dots,X_{\tau_{\ell}},X_{\tau}), U_{\tau}e_{b} \rangle U_{\tau}e_{b} \right ) \, d\tau \nonumber \\
&\quad = (\dot{g}+2\mathrm{Rc})_{\pi(U_{\tau})} \left (\nabla^{(\ell+1)}f_{\tau}(X_{\tau_{1}},\dots,X_{\tau_{\ell}},X_{\tau}), U_{\tau}e_{a} \right )\, d\tau \nonumber \\
&\quad =  (\dot{g}+2\mathrm{Rc})|_{t=T-\tau} (P_{\tau}^{-1}\nabla_{\tau}^{||}F_{\tau}, P_{\tau}^{-1}U_{0}e_{a}) \, d\tau \nonumber \\
&\quad = \langle (\dot{g}+2\mathrm{Rc})_{\tau}(\nabla_{\tau}^{||}F_{\tau})\, d\tau, U_{0}e_{a} \rangle
\end{align}
which completes the proof.
\end{proof}
\begin{theorem} (Generalized Bochner Formula on $P\mathcal{M}$) \label{3.4}
Let $F_{\tau} : P_{(x,T)}\mathcal{M} \rightarrow \R$ be a martingale. If $\sigma \geq 0$ is fixed, then $\nabla_{\sigma}^{||}F_{\tau}: P_{(x,T)}\mathcal{M} \rightarrow (T_{x}M,g_{T})$ satisfies \begin{align}
d(|\nabla_{\sigma}^{||}F_{\tau}|^2) &= \langle \nabla_{\tau}^{||}|\nabla_{\sigma}^{||}F_{\tau}|^2,dW_{\tau} \rangle + (\dot{g}+2\mathrm{Rc})_{\tau}(\nabla_{\tau}^{||}F_{\tau},\nabla_{\sigma}^{||}F_{\tau})\, d\tau \nonumber \\
&\quad + 2|\nabla_{\tau}^{||}\nabla_{\sigma}^{||}F_{\tau}|^2\, d\tau + 2|\nabla_{\sigma}^{||}F_{\sigma}|^2d\delta_{\sigma}(\tau),
\end{align}
where $(\dot{g}+2\mathrm{Rc})_{\tau}(v,w) = (\dot{g_{t}}+2\mathrm{Rc}_{g_{t}})|_{t=T-\tau}(P_{\tau}^{-1}v,P_{\tau}^{-1}w)$ and $\dot{g} = \frac{d}{dt}g$.
\end{theorem}
\begin{proof}
As $F_{\tau}$ is $\Sigma_{\tau}$-measurable, we have that $\nabla_{\sigma}^{||}F_{\tau} \equiv 0$ for $\sigma>\tau$. Noting $d(\nabla_{\sigma}^{||}F_{\tau})$ is continuous except for the jump discontinuity at $\sigma = \tau$, we calculate
\begin{align}
d|\nabla_{\sigma}^{||}F_{\tau}|^2 &= 2 \langle \nabla_{\sigma}^{||}F_{\tau}, d(\nabla_{\sigma}^{||}F_{\tau}) \rangle \nonumber \\
&= 2 \big \langle \nabla_{\sigma}^{||}F_{\tau}, d(\nabla_{\sigma}^{||}F_{\tau}) \rangle_{\mathrm{cont}} + \left (\nabla_{\sigma}^{||}F_{\sigma^{+}} - \nabla_{\sigma}^{||}F_{\sigma^{-}}  \right ) \, d\delta_{\sigma}(\tau) \big \rangle \nonumber \\
&= d(|\nabla_{\sigma}^{||}F_{\tau}|^2)_{\mathrm{cont}} + 2|\nabla_{\sigma}^{||}F_{\sigma}|^2 \, d\delta_{\sigma}(\tau).
\end{align}
It remains to show that the identity holds for $\sigma \leq \tau$. In particular, it remains to show that the continuous parts of the measures agree.
\newline \\
In the rightly-continuous case by It\^o calculus and Theorem \ref{3.3}, we calculate the quadratic variation  $d[\nabla_{\sigma}^{||}F_{\tau}, \nabla_{\sigma}^{||}F_{\tau}]=2|\nabla_{\tau}^{||}\nabla_{\sigma}^{||}F_{\tau}|^2\, d\tau$ for $\sigma \leq \tau$ and then
\begin{align}
d(|\nabla_{\sigma}^{||}F_{\tau}|^2) &= 2\langle \nabla_{\sigma}^{||}F_{\tau}, d(\nabla_{\sigma}^{||}F_{\tau}) \rangle + d[\nabla_{\sigma}^{||}F_{\tau}, \nabla_{\sigma}^{||}F_{\tau}] \nonumber \\
&= \langle \nabla_{\tau}^{||}|\nabla_{\sigma}^{||}F_{\tau}|^2, dW_{\tau} \rangle + (\dot{g}+2\mathrm{Rc})_{\tau}(\nabla_{\tau}^{||}F_{\tau},\nabla_{\sigma}^{||}F_{\tau})\, d\tau \\
&\quad + 2|\nabla_{\tau}^{||}\nabla_{\sigma}^{||}F_{\tau}|^2\, d\tau \nonumber,
\end{align}
which concludes the proof.
\end{proof}
\begin{corollary} (Bochner) \label{3.5}
The generalized Bochner formula on $P\mathcal{M}$ (Theorem \ref{3.4}) reduces to the standard Bochner formula (Lemma \ref{2.1}) in the case of $1$-point functions. That is, the evolution of $|\nabla H_{T-\tau_{1},T-\tau}f|^2$ for $\tau \leq \tau_{1}$ is given by
\begin{align}
&\frac{1}{2} \left (\partial_{\tau} + \Delta_{g_{T-\tau}} \right )|\nabla H_{T-\tau_{1},T-\tau} f|^{2} \nonumber \\
&\quad = |\nabla^{2} H_{T-\tau_{1},T-\tau}f|^2 + \frac{1}{2} (\dot{g}+2\mathrm{Rc})|_{t=T-\tau}(\nabla H_{T-\tau_{1},T-\tau}f,\nabla H_{T-\tau_{1},T-\tau}f).
\end{align}
\end{corollary}
\begin{proof}
Fix $\sigma=0$ in the evolution equation from Theorem \ref{3.4}. Next, we shall compute the evolution of $|\nabla_{0}^{||}F_{\tau}|^{2}$, where
\begin{align}
F_{\tau}(\gamma) :=
\begin{cases} H_{T-\tau_{1},T-\tau}f(\pi_{1}\gamma_{\tau}), \, \tau < \tau_{1} \\
f(\pi_{1}\gamma_{\tau_{1}}), \, \tau \geq \tau_{1}
\end{cases}
\end{align}
is the martingale induced by $f(\pi_{1} \gamma_{\tau_{1}})$. Then, for $\tau \in [0,\tau_{1}]$, we calculate
\begin{align}
|\nabla_{0}^{||}F_{\tau}|(\gamma) &= |\nabla_{\tau}^{||}F_{\tau}|(\gamma) = |\nabla H_{T-\tau_{1},T-\tau} f|(\pi_{1} \gamma_{\tau})
\end{align}
as well as
\begin{align}
|\nabla_{0}^{||}\nabla_{\tau}^{||}F_{\tau}|(\gamma) &= |\nabla^2 H_{T-\tau_{1},T-\tau}f|(\pi_{1}\gamma_{\tau}).
\end{align}
By Theorem \ref{3.4}, we then deduce that
\begin{align} \label{97}
&d(|\nabla H_{T-\tau_{1},T-\tau}f|^{2}) - \langle \nabla_{\tau}^{||} |\nabla H_{T-\tau_{1},T-\tau}|^{2}, dW_{\tau} \rangle \nonumber \\
&\quad = 2|\nabla^{2} H_{T-\tau_{1},T-\tau}f|^2\, d\tau + (\dot{g}+2\mathrm{Rc})|_{t=T-\tau}(\nabla H_{T-\tau_{1},T-\tau}f,\nabla H_{T-\tau_{1},T-\tau}f)\, d\tau
\end{align}
Moreover, for process $X_{\tau}=|\nabla H_{T-\tau_{1},T-\tau} f|^{2}(\pi_{1} \gamma_{\tau})$, by applying It\^o calculus as in Proposition \ref{1.9}, we have that
\begin{align} \label{98}
&d(|\nabla H_{T-\tau_{1},T-\tau}f|^{2}) - \langle \nabla_{\tau}^{||} |\nabla H_{T-\tau_{1},T-\tau}|^{2}, dW_{\tau} \rangle \nonumber \\
&\quad = \left (\partial_{\tau} + \Delta_{g_{T-\tau}} \right )|\nabla H_{T-\tau_{1},T-\tau}f|^{2} \, d\tau.
\end{align}
Therefore, by comparing equations \eqref{97} and \eqref{98}, we conclude that 
\begin{align}
&\frac{1}{2} \left (\partial_{\tau} + \Delta_{g_{T-\tau}} \right )|\nabla H_{T-\tau_{1},T-\tau}f|^{2} \nonumber \\
&\quad = |\nabla^{2} H_{T-\tau_{1},T-\tau}f|^2 + \frac{1}{2} (\dot{g}+2\mathrm{Rc})|_{t=T-\tau}(\nabla H_{T-\tau_{1},T-\tau}f,\nabla H_{T-\tau_{1},T-\tau}f),
\end{align}
which completes the proof.
\end{proof}

\pagebreak

\section{Applications of the Bochner Formula on Parabolic Path Space}

We shall now proceed by applying the Bochner formula on path space to both characterize the Ricci flow and develop gradient and Hessian estimates for martingales on parabolic path space.

\subsection{Proof of the Bochner Inequality on Parabolic Path Space}

\begin{proof}[Proof of Theorem \ref{theorem char}]
Using the formalism developed in the last section, we shall prove the equivalencies between the main estimates that characterize the Ricci flow.
\newline \\
$\mathrm{(R1)} \implies (\mathrm{C1}) \implies (\mathrm{C2}) \implies (\mathrm{C3})$:  If $(M,g_{t})_{t \in I}$ evolves by Ricci flow $\partial_{t}g_{t} = -2\mathrm{Rc}_{g_{t}}$ and $F_{\tau}: P_{(x,T)}\mathcal{M} \rightarrow \R$ is a martingale on parabolic path space, then Theorem \ref{Bochner gen} gives
\begin{align}
d|\nabla_{\sigma}^{||}F_{\tau}|^{2} = \langle \nabla_{\tau}^{||}|\nabla_{\sigma}^{||}F_{\tau}|^{2}, dW_{\tau} \rangle + 2|\nabla_{\tau}^{||}\nabla_{\sigma}^{||}F_{\tau}|^{2}\, d\tau + 2|\nabla_{\sigma}^{||}F_{\sigma}|^{2} d\delta_{\sigma}(\tau),
\end{align}
thus proving $(\mathrm{C1})$.
\newline \\
Next, to show $(\mathrm{C2})$, calculate
\begin{align}
|\Delta_{\sigma,\tau}^{||}F_{\tau}|^{2} &= \left |g^{ij}\left (\nabla_{\sigma}^{||}\nabla_{\tau}^{||}F_{\tau} \right )_{ij} \right |^{2} \leq |g^{ij}|^{2} |\nabla_{\sigma}^{||}\nabla_{\tau}^{||}F_{\tau}|^{2} = n|\nabla_{\sigma}^{||}\nabla_{\tau}^{||}F_{\tau}|^{2},
\end{align}
and finally show $(\mathrm{C3})$ by simply dropping the non-negative term $\frac{2}{n}|\Delta_{\sigma,\tau}^{||}F_{\tau}|^{2}$ in $(\mathrm{C2})$.
\newline \\
$\mathrm{(C1)} \implies (\mathrm{C4}) \iff (\mathrm{C5})$: To prove $(\mathrm{C4})$, first apply It\^o's lemma to the left-hand side of the full Bochner inequality $\mathrm{(C1)}$ to get
\begin{align}
&2|\nabla_{\sigma}^{||}F_{\tau}| \langle \nabla_{\tau}^{|
|}|\nabla_{\sigma}^{||}F_{\tau}|, dW_{\tau} \rangle + 2|\nabla_{\tau}^{||}\nabla_{\sigma}^{||}F_{\tau}|^2\, d\tau + 2|\nabla_{\sigma}^{||}F_{\sigma}|^2 d\delta_{\sigma}(\tau) \nonumber \\
&\quad = \langle \nabla_{\tau}^{|
|}|\nabla_{\sigma}^{||}F_{\tau}|^2, dW_{\tau} \rangle + 2|\nabla_{\tau}^{||}\nabla_{\sigma}^{||}F_{\tau}|^2\, d\tau + 2|\nabla_{\sigma}^{||}F_{\sigma}|^2 \delta_{\sigma}(\tau) \nonumber \\
&\quad \overset{\mathrm{(C1)}}{\leq} d|\nabla_{\sigma}^{||}F_{\tau}|^{2} \nonumber \\
&\quad = 2|\nabla_{\sigma}^{||}F_{\tau}|\, d|\nabla_{\sigma}^{||}F_{\tau}| + d \left [|\nabla_{\sigma}^{||}F_{\tau}|,|\nabla_{\sigma}^{||}F_{\tau}| \right ]_{\tau} \nonumber \\
&\quad = 2|\nabla_{\sigma}^{||}F_{\tau}|\, d|\nabla_{\sigma}^{||}F_{\tau}| + 2|\nabla_{\tau}^{||}|\nabla_{\sigma}^{||}F_{\tau}||^{2}\, d\tau  \nonumber \\
&\quad \leq 2|\nabla_{\sigma}^{||}F_{\tau}|\, d|\nabla_{\sigma}^{||}F_{\tau}| + 2|\nabla_{\tau}^{||}\nabla_{\sigma}^{||}F_{\tau}|^{2}\, d\tau.
\end{align}
Rearranging this inequality and applying $\mathrm{(C1)}$, we derive $\mathrm{(C4)}$, namely
\begin{align}
d|\nabla_{\sigma}^{||}F_{\tau}| \geq \langle \nabla_{\tau}|\nabla_{\sigma}^{||}F_{\tau}|,dW_{\tau} \rangle + |\nabla_{\sigma}^{||}F_{\sigma}|d\delta_{\sigma}(\tau).
\end{align}
Finally, $\mathrm{(C4)}$ is satisfied if and only if $F_{\tau}$ is a submartingale (cf. Theorem \ref{3.2}) $(\mathrm{C5})$ also holds. The remaining equivalencies will be proved in tandem with the results in the subsequent few theorems.
\end{proof}

\subsection{Proof of Gradient Estimates for Martingales}

\begin{proof}[Proof of Theorem \ref{theorem grad}]
$(\mathrm{C5}) \implies (\mathrm{G1}) \implies (\mathrm{G2})$: The implication of $(\mathrm{G1})$ follows from the definition that if $\tau \rightarrow |\nabla_{\sigma}^{||}F_{\tau}|$ is a submartingale for every $\sigma \geq 0$, then for $\tilde{\tau}\geq \tau$, $|\nabla_{\sigma}^{||}F_{\tau}| \leq \mathbb{E} \left [|\nabla_{\sigma}^{||}F_{\tilde{\tau}}| \big | \Sigma_{\tau} \right ]$. Finally, to prove $(\mathrm{G2})$, apply $\mathrm{(G1)}$ and Cauchy-Schwarz to get
\begin{align}
|\nabla_{\sigma}^{||}F_{\tau}|^{2} \leq \left (\mathbb{E} \left [|\nabla_{\sigma}^{||}F_{\tilde{\tau}}| \big | \Sigma_{\tau} \right ] \right )^{2} \leq \mathbb{E} \left [|\nabla_{\sigma}^{||}F_{\tilde{\tau}}|^{2} \big | \Sigma_{\tau} \right ] \cdot \mathbb{E}[1 \big | \Sigma_{\tau}] = \mathbb{E} \left [|\nabla_{\sigma}^{||}F_{\tilde{\tau}}|^{2} \big | \Sigma_{\tau} \right ].
\end{align}
The converse implications shall be proven along with later results.
\end{proof}

\subsection{Proof of Hessian Estimates for Martingales}

\begin{proof}[Proof of Theorem \ref{theorem hess}]
$(\mathrm{C1}) \implies (\mathrm{H1})$: To prove $(\mathrm{H1})$, fix $\sigma \geq 0$ and then integrate $(\mathrm{C1})$ from $0$ to $T$ as well as take expectations
\begin{align}
&\mathbb{E}_{(x,T)} \left [|\nabla_{\sigma}^{||}F|^{2} \right ] - \mathbb{E}_{(x,T)} \left [|\nabla_{\sigma}^{||}F_{\sigma}|^{2}| \right ] \nonumber \\
&\quad \overset{(\mathrm{C1})}{\geq} \mathbb{E}_{(x,T)} \left [\int_{0}^{T} \langle \nabla_{\tau}|\nabla_{\sigma}^{||}F_{\tau}|^{2},dW_{\tau} \rangle \right ] + 2\mathbb{E}_{(x,T)} \left [\int_{0}^{T} |\nabla_{\tau}^{||}\nabla_{\sigma}^{||}F_{\tau}|^{2}\, d\tau \right ] \nonumber \\
&\quad = 2\mathbb{E}_{(x,T)} \left [\int_{0}^{T} |\nabla_{\tau}^{||}\nabla_{\sigma}^{||}F_{\tau}|^{2}\, d\tau \right ].
\end{align}
\noindent $(\mathrm{H1}) \implies (\mathrm{H2})$: To prove $(\mathrm{H2})$, apply It\^o isometry and then integrate $(\mathrm{H1})$ from $0$ to $T$ with respect to $\sigma$ as well as take expectations
\begin{align}
\mathbb{E}_{(x,T)} \left [(F - \mathbb{E}_{(x,T)}[F])^{2} \right ] &= \mathbb{E}_{(x,T)} \left [\int_{0}^{T} |\nabla_{\sigma}^{||}F_{\sigma}|^{2} \, d\sigma \right ] \nonumber \\
&\overset{(\mathrm{H1})}{\leq} \mathbb{E}_{(x,T)} \left [\int_{0}^{T} |\nabla_{\sigma}^{||}F|^{2}\, d\sigma \right ] \\
&\quad - 2\mathbb{E}_{(x,T)} \left [\int_{0}^{T} \int_{0}^{T} |\nabla_{\tau}^{||}\nabla_{\sigma}^{||}F_{\tau}|^{2}\, d\tau\, d\sigma \right ]. \nonumber 
\end{align}
\noindent $(\mathrm{R1}) \implies (\mathrm{H3})$: To prove $(\mathrm{H3})$, let $G=F^2$ and consider the evolution equation for $X_{\tau} := G_{\tau}^{-1}|\nabla^{\mathcal{H}}G_{\tau}|^{2}-2G_{\tau}\log(G_{\tau})$, which satisfies
\begin{align}
dX_{\tau} &= \langle \nabla_{\tau}^{||}X_{\tau}, dW_{\tau} \rangle + 2G_{\tau} \left (\int_{0}^{T} |\nabla_{\tau}^{||}\nabla_{\sigma}^{||}\log(G_{\tau})|^{2}\, d\sigma \right )\, d\tau \nonumber \\
&\quad + G_{\tau}^{-1} \left (\int_{0}^{T} (\dot{g}+2\mathrm{Rc})_{\tau}(\nabla_{\tau}^{||}F_{\tau},\nabla_{\sigma}^{||}F_{\tau}) \, d\sigma \right )\, d\tau \nonumber \\
&\geq \langle \nabla_{\tau}^{||}X_{\tau}, dW_{\tau} \rangle + 2G_{\tau} \left (\int_{0}^{T} |\nabla_{\tau}^{||}\nabla_{\sigma}^{||}\log(G_{\tau})|^{2}\, d\sigma \right )\, d\tau \label{136}
\end{align}
by It\^o calculus and Proposition \ref{3.3} (cf. Proposition 4.23 of \cite{HN18b}). Next, integrate the inequality \eqref{136} from $0$ to $T$ with respect to $\tau$ and take expectations to get
\begin{align}
\mathbb{E}_{(x,T)}[X_{T}] - \mathbb{E}_{(x,T)}[X_{0}] \geq 2 \mathbb{E}_{(x,T)} \left [G_{\tau} \left (\int_{0}^{T} |\nabla_{\tau}^{||}\nabla_{\sigma}^{||}\log(G_{\tau})|^{2}\, d\sigma \right )\, d\tau \right ],
\end{align}
and evaluating the two expectations in the difference, namely
\begin{align}
\mathbb{E}_{(x,T)}[X_{0}] &= \mathbb{E}_{(x,T)}[G_{0}^{-1} |\nabla^{\mathcal{H}}G_{0}|^2 - 2G_{0}\log(G_{0})] \nonumber \\
&= 0 - 2G_{0}\log(G_{0}) \nonumber \\
&= -2\mathbb{E}_{(x,T)}[F^2]\log \left (\mathbb{E}_{(x,T)}[F^2]\right )
\end{align}
and 
\begin{align}
\mathbb{E}_{(x,T)}[X_{T}] &= \mathbb{E}_{(x,T)} \left [G^{-1}|\nabla^{\mathcal{H}}G|^2-2G\log(G) \right ] \nonumber \\
&= \mathbb{E}_{(x,T)}\left [F^{-2}|\nabla^{\mathcal{H}}F^2|^2 \right ] - 2\mathbb{E}_{(x,T)} \left [F^{2}\log(F^{2}) \right ] \nonumber \\
&= 4\mathbb{E}_{(x,T)} \left [|\nabla^{\mathcal{H}}F|^2 \right ]  - 2\mathbb{E}_{(x,T)} \left [F^{2}\log(F^{2}) \right ].
\end{align}
Finally we observe that
\begin{align}
\mathbb{E}_{(x,T)} \left [|\nabla^{\mathcal{H}}F|^2 \right ] = \mathbb{E}_{(x,T)} \left [\int_{0}^{T} |\nabla_{\sigma}^{||}F|^2\, d\sigma \right ],
\end{align}
and then combine this and the aforementioned results to prove the claim.
\end{proof}

\subsection{Proof of the Characterizations of Solutions of the Ricci Flow}

The following result reproves a theorem by Haslhofer and Naber (cf. Theorem 1.22 of \cite{HN18a}), characterizing solutions of the Ricci flow, using the Bochner formulas on path space that were developed in the previous section.
\begin{proof}[Proof of Theorem \ref{theorem Ricci}]
$(\mathrm{G1}) \implies (\mathrm{R2})$: To prove $(\mathrm{R2})$, we evaluate $(\mathrm{G1})$ at $\sigma=\tau=0$,
\begin{align}
\left |\nabla_{x} \mathbb{E}_{(x,t)}[F] \right | = \left |\nabla_{x}F_{0} \right | \overset{(\mathrm{G1})}{\leq} \mathbb{E}_{(x,T)} \left [|\nabla_{0}^{||}F| \big | \Sigma_{0} \right ] = \mathbb{E}_{(x,T)} \left [|\nabla_{0}^{||}F| \right ].
\end{align}
$(\mathrm{G2}) \implies (\mathrm{R3})$: To prove $(\mathrm{R3})$, we evaluate $(\mathrm{G2})$ at $\sigma,\tau = 0$ (and observe that $d[F,F]_{\tau} = 2|\nabla_{\tau}^{||}F_{\tau}|^{2}\, d\tau$ by Theorem \ref{3.2}),
\begin{align}
\mathbb{E}_{(x,T)} \left [\frac{d[F,F]_{\tau}}{d\tau} \right ] &= 2 \mathbb{E}_{(x,T)} \left [|\nabla_{\tau}^{||}F_{\tau}|^{2} \right ] \nonumber \\
&\overset{(\mathrm{G2})}{\leq} 2 \mathbb{E}_{(x,T)} \left [\mathbb{E}_{(x,T)} \left [|\nabla_{\tau}^{||}F|^{2}  \big | \Sigma_{\tau} \right ] \right ] \nonumber \\
&\hspace{0.3em} \leq 2 \mathbb{E}_{(x,T)} \left [|\nabla_{\tau}^{||}F|^{2} \right ]. 
\end{align}
$(\mathrm{G1}) \implies (\mathrm{R4})$: To prove $(\mathrm{R4})$, we set $\sigma=\tau$ and take expectations
\begin{align}
\mathbb{E}_{(x,T)} \left [|\nabla_{\sigma}^{||}F_{\sigma}| \right ] \leq \mathbb{E}_{(x,T)} \left [\mathbb{E}_{(x,T)} \left [|\nabla_{\sigma}^{||}F| \big | \Sigma_{\sigma} \right ] \right ] = \mathbb{E}_{(x,T)} \left [|\nabla_{\sigma}^{||}F| \right ].
\end{align}
Then follow the proof of $(\mathrm{H3})$ in Theorem \ref{theorem hess} and evaluate the expectation $\mathbb{E}_{(x,T)}[X_{\tau_{j}}]$ for $j \in \{1,2\}$, namely
\begin{align}
\mathbb{E}_{(x,T)}[X_{\tau_{j}}] = \mathbb{E}_{(x,T)} \left [G_{\tau_{j}}^{-1}|\nabla^{\mathcal{H}}G_{\tau_{j}}|^{2} \right ] - 2\mathbb{E}_{(x,T)} \left [G_{\tau_{j}}\log(G_{\tau_{j}}) \right ]
\end{align}
and taking differences, where $\mathbb{E}_{(x,T)}[X_{\tau_{2}}\, |\, \Sigma_{\tau_{1}}]-\mathbb{E}_{(x,T)}[X_{\tau_{1}}\, |\, \Sigma_{\tau_{1}}] \geq 0$, as in the earlier proof. It remains to check that
\begin{align}
&\mathbb{E}_{(x,T)}\left [G_{\tau_{2}}^{-1}|\nabla^{\mathcal{H}}G_{\tau_{2}}|^{2}\, |\, \Sigma_{\tau_{1}} \right ] - \mathbb{E}_{(x,T)}\left [G_{\tau_{1}}^{-1}|\nabla^{\mathcal{H}}G_{\tau_{1}}|^{2}\, |\, \Sigma_{\tau_{1}} \right ] \nonumber \\
&\quad = 4\mathbb{E}_{(x,T)} \left [|\nabla^{\mathcal{H}}F_{\tau_{2}}|^{2}\, |\, \Sigma_{\tau_{1}} \right ] \nonumber \\
&\quad = 4\mathbb{E}_{(x,T)} \left [\int_{\tau_{1}}^{\tau_{2}} |\nabla_{\sigma}^{||}F_{\tau_{2}}|^{2}\, d\sigma \right ] \nonumber \\
&\quad \leq 4\mathbb{E}_{(x,T)} \left [\int_{\tau_{1}}^{\tau_{2}} |\nabla_{\sigma}^{||}F|^{2}\, d\sigma \right ]\quad (\tau \rightarrow |\nabla_{\sigma}^{||}F_{\tau}|^{2} \mathrm{\, is\, a\, submartingale}) \nonumber \\
&\quad = 4\mathbb{E}_{(x,T)} \left [\langle F, \mathcal{L}_{(\tau_{1},\tau_{2})}F \rangle_{\mathcal{H}} \right ].
\end{align}
$(\mathrm{G2}) \implies (\mathrm{R5})$: To prove $(\mathrm{R5})$, we set $\sigma = \tau$ and take expectations
\begin{align}
\mathbb{E}_{(x,T)} \left [|\nabla_{\sigma}^{||}F_{\sigma}|^{2} \right ] \leq \mathbb{E}_{(x,T)} \left [\mathbb{E}_{(x,T)} \left [|\nabla_{\sigma}^{||}F|^{2}  \big | \Sigma_{\sigma} \right ] \right ] = \mathbb{E}_{(x,T)} \left [|\nabla_{\sigma}^{||}F|^{2} \right ].
\end{align}
Then follow the proof of $(\mathrm{H2})$ in Theorem \ref{theorem hess} and apply It\^o isometry
\begin{align}
\mathbb{E}_{(x,T)} \left [(F_{\tau_{2}}-F_{\tau_{1}})^{2}\, |\, \Sigma_{\tau_{1}} \right ] &= \mathbb{E}_{(x,T)} \left [\int_{0}^{T} |\nabla_{\sigma}^{||}F_{\tau_{2}}|^{2}\, d\sigma\, \bigg |\, \Sigma_{\tau_{1}} \right ] \nonumber \\
&= \mathbb{E}_{(x,T)} \left [\int_{\tau_{1}}^{\tau_{2}} |\nabla_{\sigma}^{||}F_{\tau_{2}}|^{2}\, d\sigma \right ] \nonumber \\
&\leq \mathbb{E}_{(x,T)} \left [\langle F, \mathcal{L}_{(\tau_{1},\tau_{2})}F \rangle_{\mathcal{H}} \right ].
\end{align}
The converse implications shall be proven in the next section.
\end{proof}

\subsection{Converse Implications}

We shall now prove the converse implications below.

\begin{proof}
$(\mathrm{C3}) \implies (\mathrm{R1})$: First fix $(x,T) \in \mathcal{M}$ and $v\in (T_{x}M,g_{T})$ a unit vector and choose a smooth compactly supported $f_{1}: M \rightarrow \R$ such that
\begin{align}
f_{1}(x)=0, \quad \nabla f_{1}(x)=v, \quad \nabla^{2} f_{1}(x)=0
\end{align}
using exponential coordinates. Consider the one-point cylinder function given by $F(\gamma)=f_{1}(\pi_{M}(\gamma(\varepsilon)))$, $F: P_{(x,T)}\mathcal{M} \rightarrow \R$ and observe for $\tau \leq \varepsilon$ that
\begin{align}
\nabla_{\tau}^{||}F_{\tau} = P_{\tau}\nabla H_{T-\tau,T}f_{1}(\pi_{M}(\gamma(\tau))), \quad |\nabla_{\tau}^{||}\nabla_{0}^{||}F_{\tau}| = |\nabla^{2}H_{T-\tau,T}f_{1}|(\pi_{M}(\gamma(\tau))).
\end{align}
In particular, $\nabla_{\tau}^{||}F_{\tau} = v + o(\varepsilon)$ and $|\nabla_{\tau}^{||}\nabla_{0}^{||}F_{\tau}| = o(\varepsilon)$. Then, by Theorem \ref{3.4}, 
\begin{align}
\tau \rightarrow |\nabla_{0}^{||}F_{\tau}|^{2} - \int_{0}^{\tau} \left (2|\nabla_{\rho}^{||}\nabla_{0}^{||}F_{\rho}|^{2}+(\dot{g}+2\mathrm{Rc})_{\rho}(\nabla_{\rho}^{||}F_{\rho},\nabla_{0}^{||}F_{\rho}) \right )\, d\rho
\end{align}
is a martingale. So, in particular,
\begin{align}
|\nabla_{0}^{||}F_{0}|^{2} = \mathbb{E} \left [|\nabla_{0}^{||}F_{\varepsilon}|^{2}\right ] - \varepsilon (\dot{g}+2\mathrm{Rc})_{\varepsilon}(v,v) + o(\varepsilon).
\end{align}
Moreover, since $\tau \rightarrow |\nabla_{0}^{||}F_{\tau}|^{2}$ is a submartingale by $(\mathrm{C3})$, it follows that
\begin{align} \label{165}
(\dot{g}+2\mathrm{Rc})_{\varepsilon}(v,v) \geq \varepsilon^{-1}o(\varepsilon).
\end{align}
Next choose a smooth compactly supported $f_{2}: M\times M \rightarrow \R$ such that
\begin{align}
f_{2}(x,x)=0, \quad \nabla^{(1)}f_{2}(x,x)=2v, \quad \nabla^{(2)}f_{2}(x,x)=-v, \quad \nabla^{2}f_{2}(x,x)=0,
\end{align}
for example $f_{2}(y,z)=2f_{1}(y)-f_{1}(z)$. Consider the two-point cylinder function given by $F(\gamma)=f_{2}(\pi_{M}(\gamma(0)),\pi_{M}(\gamma(\varepsilon)))$, $F: P_{(x,T)}\mathcal{M} \rightarrow \R$ and observe for $\tau \leq \varepsilon$ that
\begin{align}
\begin{cases}
\nabla_{0}^{||}F_{\tau} &= \nabla^{(1)}f_{2}(x,\pi_{M}(\gamma(\tau))) + P_{\tau}\nabla H_{T-\tau,T}^{(2)}f_{2}(x,\pi_{M}(\gamma(\tau))) \\
\nabla_{\tau}^{||}F_{\tau} &= P_{\tau}\nabla H_{T-\tau,T}^{(2)}f_{2}(x,\pi_{M}(\gamma(\tau))) \\
|\nabla_{\tau}^{||}\nabla_{0}^{||}F_{\tau}| &\leq |\nabla^{2}f_{2}|(x,\pi_{M}(\gamma(\tau))) + |\nabla^{2}H_{T-\tau,T}^{(2)}f_{2}|(x,\pi_{M}(\gamma(\tau))).
\end{cases}
\end{align}
In particular,  $\nabla_{0}^{||}F_{\tau} = v + o(\varepsilon)$, $\nabla_{\tau}^{||}F_{\tau} = -v + o(\varepsilon)$ and $|\nabla_{\tau}^{||}\nabla_{0}^{||}F_{\tau}| = o(\varepsilon)$. Then, again by Theorem \ref{3.4},
\begin{align}
|\nabla_{0}^{||}F_{0}|^{2} = \mathbb{E} \left [|\nabla_{0}^{||}F_{\varepsilon}|^{2}\right ] + \tau(\dot{g}+2\mathrm{Rc})_{\varepsilon}(v,v) + o(\varepsilon).
\end{align}
Moreover, since $\tau \rightarrow |\nabla_{0}^{||}F_{\tau}|^{2}$ is a submartingale by $(\mathrm{C3})$, it follows that 
\begin{align} \label{171}
(\dot{g}+2\mathrm{Rc})_{\varepsilon}(v,v) \leq \tau^{-1}o(\varepsilon).
\end{align}
We can then deduce that $(\mathrm{R1})$ is satisfied by taking $\varepsilon \rightarrow 0^{+}$ in equations \eqref{165} and \eqref{171}.
\newline \\
To check the remaining converse implications, one can substitute 1-point and 2-point cylinder functions as above. However, there are some alternative tools that can close the loop of equivalencies more readily. For example, applying the log-Sobolev equality to $F^2=1+\varepsilon G$ in $(\mathrm{R4})$ gives the Poincar\'e inequality in $(\mathrm{R5})$. Moreover, dividing by $T-\tau$, taking $T-\tau \rightarrow 0^{+}$ and using the quadratic variation $d[F,F]_{\tau} = 2|\nabla_{\tau}^{||}F_{\tau}|^{2}\, d\tau$ (by Theorem \ref{3.2}), $(\mathrm{R3})$ can be derived from $(\mathrm{R5})$. In short, some implications can be done directly without the need to appeal to test functions each time.
\end{proof}

\pagebreak

\bibliographystyle{plain}

\noindent \textsc{Christopher Kennedy, Department of Mathematics, University of Toronto, 40 St. George Street, Toronto, ON M5S 2E4, Canada}
\newline \\
email: christopherpa.kennedy@mail.utoronto.ca

% --------------------------------------------------------------
%     You don't have to mess with anything below this line.
% --------------------------------------------------------------

\end{document}